\documentclass[onecolumn,reqno]{amsart}
\usepackage{amsthm}
\usepackage{times,amssymb,amsmath,amsfonts,float,nicefrac,bbm,mathrsfs,caption,float}
\usepackage{enumerate,multirow,caption,tikz,graphicx,enumitem,nicefrac,cases}
\usepackage[mathscr]{eucal}
\usepackage{sidecap}
\usepackage{algpseudocode}
\usepackage{verbatim}
\usepackage{epstopdf}
\usepackage{textcomp}
\usetikzlibrary{shapes,arrows}
\usepackage[normalem]{ulem}
\usepackage{mathabx}
\usepackage[top=1.5in,bottom=1.43in,left=1.25in,right=1.25in]{geometry}
\allowdisplaybreaks

\usepackage{hyperref} % ref online resource
\hypersetup{
    colorlinks,
    linkcolor={red!80!black},
    citecolor={red!80!black},
    urlcolor={blue!80!black}
}
\usepackage[font={it}]{caption}
\usepackage[margin=0.05in]{subcaption}
\usepackage[bottom]{footmisc}

\usepackage[normalem]{ulem}
\usepackage[ruled,vlined,linesnumbered]{algorithm2e}

\newcommand\nc\newcommand
\nc\bfa{{\boldsymbol a}}\nc\bfA{{\boldsymbol A}}\nc\cA{{\mathscr A}}
\nc\bfb{{\boldsymbol b}}\nc\bfB{{\boldsymbol B}}\nc\cB{{\mathscr B}}
\nc\bfc{{\boldsymbol c}}\nc\bfC{{\boldsymbol C}}\nc\cC{{\mathscr C}}
\nc\bfd{{\boldsymbol d}}\nc\bfD{{\boldsymbol D}}\nc\cD{{\mathscr D}}
\nc\bfe{{\boldsymbol e}}\nc\bfE{{\boldsymbol E}}\nc\cE{{\mathscr E}}
\nc\bff{{\boldsymbol f}}\nc\bfF{{\boldsymbol F}}\nc\cF{{\mathscr F}}
\nc\bfg{{\boldsymbol g}}\nc\bfG{{\boldsymbol G}}\nc\cG{{\mathscr G}}
\nc\bfh{{\boldsymbol h}}\nc\bfH{{\boldsymbol H}}\nc\cH{{\mathscr H}}
\nc\bfi{{\boldsymbol i}}\nc\bfI{{\boldsymbol I}}\nc\cI{{\mathcal I}}
\nc\bfj{{\boldsymbol j}}\nc\bfJ{{\boldsymbol J}}\nc\cJ{{\mathscr J}}
\nc\bfk{{\boldsymbol k}}\nc\bfK{{\boldsymbol K}}\nc\cK{{\mathscr K}}
\nc\bfl{{\boldsymbol l}}\nc\bfL{{\boldsymbol L}}\nc\cL{{\mathscr L}}
\nc\bfm{{\boldsymbol m}}\nc\bfM{{\boldsymbol M}}\nc{\cM}{{\mathscr M}}
\nc\bfn{{\boldsymbol n}}\nc\bfN{{\boldsymbol N}}\nc\cN{{\mathscr N}}
\nc\bfo{{\boldsymbol o}}\nc\bfO{{\boldsymbol O}}\nc\cO{{\mathscr O}}
\nc\bfp{{\boldsymbol p}}\nc\bfP{{\boldsymbol P}}\nc\cP{{\mathscr P}}\nc\eP{{\EuScript P}}\nc\fP{{\mathfrak P}}
\nc\bfq{{\boldsymbol q}}\nc\bfQ{{\boldsymbol Q}}\nc\cQ{{\mathscr Q}}
\nc\bfr{{\boldsymbol r}}\nc\bfR{{\boldsymbol R}}\nc\cR{{\mathscr R}}
\nc\bfs{{\boldsymbol s}}\nc\bfS{{\boldsymbol S}}\nc\cS{{\mathscr S}}
\nc\bft{{\boldsymbol t}}\nc\bfT{{\boldsymbol T}}\nc\cT{{\mathscr T}}
\nc\bfu{{\boldsymbol u}}\nc\bfU{{\boldsymbol U}}\nc\cU{{\mathscr U}}
\nc\bfv{{\boldsymbol v}}\nc\bfV{{\boldsymbol V}}\nc\cV{{\mathscr V}}
\nc\bfw{{\boldsymbol w}}\nc\bfW{{\boldsymbol W}}\nc\cW{{\mathscr W}}
\nc\bfx{{\boldsymbol x}}\nc\bfX{{\boldsymbol X}}\nc\cX{{\mathscr X}}
\nc\bfy{{\boldsymbol y}}\nc\bfY{{\boldsymbol Y}}\nc\cY{{\mathscr Y}}
\nc\bfz{{\boldsymbol z}}\nc\bfZ{{\boldsymbol Z}}\nc\cZ{{\mathscr Z}}
\nc{\bb}{{\mathbbm{1}}}
\nc\reals{{\mathbb R}}
\nc{\half}{{\nicefrac12}}

\newtheorem{theorem}{Theorem}[section]

\newtheorem{proposition}[theorem]{Proposition}

\theoremstyle{remark}
\newtheorem{remark}{Remark}[section]

\DeclareMathOperator{\SRG}{\text {SRG}}

\usepackage{xargs}                      % Use more than one optional parameter in a new commands
\usepackage{xcolor}
\usepackage[colorinlistoftodos,prependcaption,textsize=footnotesize,textwidth=1in]{todonotes}
\newcommandx{\unsure}[2][1=]{\todo[linecolor=red,backgroundcolor=red!25,bordercolor=red,#1]{#2}}
\newcommand\redout{\bgroup\markoverwith{\textcolor{red}{\rule[0.5ex]{2pt}{0.8pt}}}\ULon}

\newcommandx{\rednote}[2][1=]{\todo[linecolor=red,backgroundcolor=red!25,bordercolor=red,#1]{#2}}
\newcommandx{\bluenote}[2][1=]{\todo[linecolor=blue,backgroundcolor=blue!25,bordercolor=blue,#1]{#2}}
\newcommandx{\yellownote}[2][1=]{\todo[linecolor=yellow,backgroundcolor=yellow!25,bordercolor=yellow,#1]{#2}}
\newcommandx{\greennote}[2][1=]{\todo[inline,linecolor=olive,backgroundcolor=green!25,bordercolor=olive,#1]{#2}}

\newcommand{\blue}[1]{{\color{black} #1}}

\begin{document}
		\title[]{Bounds for the sum of distances of spherical sets of small size}
	\author[]{Alexander Barg$^1$, Peter Boyvalenkov$^2$, Maya Stoyanova$^3$}
	\thanks{\hspace*{-.15in}$^1$ Department of ECE and ISR, University of Maryland, College Park, MD 20742, USA. Supported in part by NSF grants CCF 1814487,
 CCF 2110113 (NSF-BSF) and CCF 2104489. Email: abarg@umd.edu.}
 \thanks{$^2$ Institute of Mathematics and Informatics, Bulgarian Academy of Sciences, 8 G Bonchev Str., 1113  Sofia, Bulgaria. 
Supported in part by Bulgarian NSF project KP-06-Russia/33-2020. Email: peter@math.bas.bg.}
\thanks{$^3$ Faculty of Mathematics and Informatics, Sofia University, 5 James Bourchier Blvd., 1164 Sofia, Bulgaria.
Supported in part by Bulgarian NSF project KP-06-N32/2-2019. Email: stoyanova@fmi.uni-sofia.bg. }
 \begin{abstract}  
We derive upper and lower bounds on the sum of distances of a spherical code of size $N$ in $n$ dimensions when $N=\Theta(n^\alpha), 0<\alpha\le 2.$ The bounds are derived by specializing recent general, universal bounds on energy of spherical sets. We discuss asymptotic behavior of our bounds along with several examples of codes whose sum of distances closely follows the upper bound. 
 \end{abstract}
 		\maketitle

\vspace{-.2in}

\section{Introduction: Sum of distances and related problems}
\subsection{Problem statement and overview of results}
Let ${\cC}_N=\{z_1,\dots,z_N\}$ be a set of $N$ points (code) on the unit sphere $S^{n-1}$ in $\reals^n$. Denote by
$\tau_n({\cC}_N)=\sum_{i,j=1}^N \|z_i-z_j\|$ the sum of pairwise distances between the points in ${\cC}_N$ and let 
$\tau(n,N)=\sup_{{\cC}_N}\tau_n({\cC}_N)$ be the largest attainable sum of distances over all sets of cardinality $N$.
The problem of estimating $\tau(n,N)$ was introduced by Fejes T{\'o}th \cite{FejesToth1956} and it has been studied in 
a large number of follow-up papers, \cite{HouShao2011,BilykMatzke2019}. The main body of results in the literature are concerned with the asymptotic
regime of fixed $n$ and $N\to\infty.$ In particular, it is known that     \begin{equation}\label{eq:old-bounds}
      cN^{1-\frac1{n-1}}\le W(S^{n-1}) N^2-\tau(n,N)\le     CN^{1-\frac1{n-1}},
    \end{equation}
where $W(S^{n-1})=\iint \|x-y\|d\sigma_n(x)d\sigma_n(y)$ is the average distance on the sphere, $\sigma_n$ is the normalized (surface area) measure on the sphere. and $c, C$ are some positive constants that depend only on $n$.
The upper bound in \eqref{eq:old-bounds} is due to Alexander \cite{Alexander1972} \blue{for $n=3$ and Stolarsky \cite{Stolarsky1973} for higher dimensions}, and the lower bound was proved by Beck \cite{Beck1984}. Kuijlaars and Saff \cite{Kuijlaars1998} extended these results to bounds on the $s$-Riesz energy of spherical sets for all $s>0$, and Brauchart et al. \cite{Brauchart2012} computed next terms of the asymptotics; see also Ch.~6 in a comprehensive monograph by Borodachov et al.~\cite{BHS2019}  for a recent overview. 

In this paper we adopt a different view, allowing both the dimension $n$ and the cardinality $N$ to increase in a certain related way.
The main emphasis of this work is on obtaining explicit lower and upper bounds on the sum of distances of a spherical set ${\cC}_N$ for
$N\sim \delta n^\alpha$, \blue{for certain $\delta$ and} $0< \alpha\le 2.$ Upper bounds apply uniformly for all spherical sets, while to derive lower bounds we need to assume that the minimum pairwise distance is bounded from below (otherwise the sum of distances can be made arbitrarily small). If the minimum distance is large, then the neighbors of a point are naturally placed \blue{on or near the orthogonal subsphere (the ``equator'')}, and the distance to them is about $\sqrt 2.$ This suggests that the main term in the asymptotic expression for the sum of distances is $\sqrt2 N^2,$ and it is easy to obtain a bound of the form $\tau(n,N)\le \sqrt2 N^2(1+o(1)),$
as shown below in Sec.~\ref{sec:discrepancy}.

Our main results are related to refinements of this claim. Using linear programming, we derive lower and upper 
bounds for the sum of distances of codes of small size. For a number of code families, the sum of distances behaves as $\sqrt2 N^2,$ and the bound is asymptotically tight. We compute lower-order terms in a number of examples, including codes obtained from equiangular line sets, spherical embeddings of strongly regular graphs (two-distance tight frames), and spherical embeddings of some classes of small-size binary codes. Numerical calculations, some of which we include, confirm that the sum of distances of these codes follows closely the upper bound.

\subsection{Sum of distances and Stolarsky's invariance}\label{sec:discrepancy}
The sum of distances in a spherical code enjoys several links with other problems in
geometry of spherical sets. One of them is related to the theory of uniform distributions on the sphere. A
sequence of spherical sets $({\cC}_N)_N$ is called asymptotically uniformly distributed if for every closed set $A\subset S^{n-1}$
     $$
     \lim_{N\to\infty}\frac{|{\cC}_N\cap A|}N=\sigma_n(A).
     $$
To quantify the proximity of a sequence of sets ${\cC}_N$ to the uniform distribution on $S^{n-1},$ define 
the {\em quadratic discrepancy}\footnote{More precisely, the discrepancy is defined as $(D^{L_2}(\cC_N))^{1/2},$ and it is called the {\em spherical cap discrepancy}, as there are also other types of discrepancy on the sphere.} of ${\cC}_N:$
   \begin{equation}\label{eq:disc}
   D^{L_2}({\cC}_N):=\int_{-1}^1 \int_{S^{n-1}}\Big|\frac 1N\sum_{j=1}^N \bb_{C(x,t)}(z_j)-\sigma_n(C(x,t))\Big|^2 d\sigma_n(x) dt,
   \end{equation}
where $C(x,t)=\{y\in S^{n-1}: (x \cdot y) \ge t\}$ is a {\em spherical cap} of radius $\arccos t$ centered at $x.$ 
A classic result
states that a sequence of sets ${\cC}_N$ is asymptotically uniformly distributed if and only if $\lim_{N\to\infty} D^{L_2}({\cC}_N)=0;$
see, e.g., \cite[Theorem 6.1.5]{BHS2019}.
A fundamental relation between $\tau_n({\cC}_N)$ and $D^{L_2}({\cC}_N)$ states that the sum of these two quantities is a constant that depends only
on $N$ and $n.$ Namely,
  \begin{equation}\label{eq:Stol}
  c_n D^{L_2}({\cC}_N)=W(S^{n-1})-\frac 1{N^2}\tau_n({\cC}_N),
  \end{equation}
where 
\blue{$c_n={(n-1)\sqrt{\pi}\Gamma((n-1)/2)}/{\Gamma(n/2)}$} is a universal constant that depends only on the dimension of the sphere.
This relation was proved by Stolarsky  \cite{Stolarsky1973} and is \blue{now} known as {\em Stolarsky's invariance principle}. 
The average distance on the sphere is given by $\int_0^\pi 2\sin(\theta/2)\sin^{n-1}\theta d\theta/\int_0^\pi \sin^{n-1}\theta d\theta,$ which evaluates to
   \begin{equation*}
   W(S^{\blue{n-1}})=\frac{2^{n-1}\Gamma(n/2)^2}{\sqrt\pi\Gamma(n-\nicefrac12)}=\sqrt 2-\frac 1{4\sqrt 2n}+O(n^{-2}).
   \end{equation*}
Since $D^{L_2}({\cC}_N)\ge 0,$ the following bound is immediate: for any code ${\cC}_N\subset S^{n-1}$
   \begin{equation}\label{eq:disc-dist}
    \tau_n({\cC}_N)\le N^2\Big(\sqrt 2-\frac 1{4\sqrt 2n}+O(n^{-2})\Big).
   \end{equation}
This inequality in effect states a well-known fact that the average of a \blue{radial} negative-definite kernel, over a subset of the sphere is at most the average over the entire \blue{sphere}. It also forms a very particular case of a recent general result in 
\blue{\cite[Theorem 3.1]{BDHSS2016}}.

{\em Remarks}

1. On account of \eqref{eq:Stol}, the problem of maximizing the sum of distances is equivalent to minimizing the quadratic discrepancy, i.e., the sum of distances serves as a proxy for uniformity: a set of $N$ points on the sphere is ``more uniform'' if the sum of pairwise distances is large for its size. 

2. Sequences $({\cC}_N)$ with average distance $\sqrt 2(1+o(1))$ are asymptotically uniformly distributed. As we have already pointed out, many
sequences of codes satisfy this condition; moreover, as shown below, spherical codes obtained from the binary Kerdock and dual BCH codes match the second term in \eqref{eq:disc-dist}, implying a faster rate of convergence to the limit.

3. Extensions and generalizations of Stolarsky's invariance were proposed in recent works
\cite{Brauchart2013, BilykLacey2017,Bilyk2018,Skriganov2019,Skriganov2020,Barg2021}. In particular, \cite{Barg2021} studied quadratic discrepancy of binary codes, deriving explicit expressions as well as some bounds. Below in Sec.~\ref{sec:binary}, we 
point out that this problem is closely related to the sum-of-distances problem in the spherical case, and translate our results on bounds to the binary case. This link also motivates studying the asymptotic regime of $n\to\infty$ for spherical codes because this is the only possible asymptotics in the binary space.

\subsection{Details of our approach}
Viewing the distance $\|x-y\|$ as a two-point potential on the sphere, we can relate the problem of estimating
$\tau(n,N)$ to the search
for spherical configurations with the minimum potential energy. References \cite{Kuijlaars1998}, \cite{Brauchart2012},
\cite{BHS2019}, and many others adopt this point of view, considering the energy minimization for general classes of potential functions on the sphere. A line of works on energy 
minimization, initiated by Yudin \cite{Yudin1993,Kolushov1997} and developed by Cohn and Kumar \cite{coh07b}, uses the linear programming  bounds on codes to derive results about optimality as well as lower bounds on the energy of spherical codes. 
Extending the approach of earlier works by Yudin and Levenshtein \cite{lev83a,lev98}, the authors of
\cite{coh07b} proved optimality of several known spherical codes for all {\em absolutely monotone} potentials\footnote{A potential $h(t):[-1,1]\to\reals$ is called absolutely monotone if for every $n\ge0$ the
derivative $h^{(n)}(t)$ exists and is nonnegative for all $t$.} and called
such codes universally optimal.  In particular, denoting $t=t(x,y)=x\cdot y,$ we immediately observe that
the potential $L(t)=-\|x-y\|=-\sqrt{2(1-t)}$ \blue{fits in this scheme since $2+L(t)$} is absolutely monotone, and thus all the known universally optimal codes are
maximizers of the sum of distances.

While the results of \cite{coh07b} apply to specific spherical codes, a suite of {\em universal bounds} on the potential energy 
was derived in recent papers of Boyvalenkov, Dragnev, Hardin, Saff, and Stoyanova~\cite{BDHSS2015,BDHSS2016,BDHSS2019,BDHSS2020}. While the bounds can be written in a general form relying on the Levenshtein formalism, explicit expressions are difficult to come by. We derive an explicit form of the first few bounds in the Levenshtein hierarchy and evaluate them for the families of spherical codes mentioned above, limiting ourselves to the potential $L(t).$
Our approach can be summarized as follows.
Given an absolutely monotone potential $h$, define the minimum $h$-energy over all spherical sets of size $N$ by 
\[ E_h(n,N):=\inf_{{\cC}_N} E_h({\cC}_N), \] 
where $E_h({\cC}_N)=\sum_{i,j=1}^N h(z_i\cdot z_j).$ This quantity is bounded from below as follows:
   \begin{equation} \label{ulb-general} 
   E_h(n,N) \geq N^2\sum_{i=0}^{k-1+\varepsilon} \rho_i h(\alpha_i), 
    \end{equation}
where the positive integer $k$, the value $\varepsilon \in \{0,1\}$, and the real parameters $(\rho_i,\alpha_i)$, $i=0,1,\ldots,k-1+\varepsilon$, 
are functions of $N$ and $n$ as explained in \cite{BDHSS2016} and in Section \ref{sec:proofs-bounds} below. 
The bound \eqref{ulb-general} was called a {\em universal lower bound} (ULB) in \cite{BDHSS2016}. 
For given $k$ and $\varepsilon$ we obtain a degree-$m$ bound, $m=2k-1+\varepsilon$, where the term ``degree'' refers to the degree
of the polynomial used in the corresponding linear programming problem.
The bound of degree $m$ applies to the values of code cardinality in the segment
$D^{\ast}(n,m)\le N< D^{\ast}(n,m+1),$ where $D^{\ast}(n,m):= \binom{n+k-2+\varepsilon}{n-1}+\binom{n+k-2}{n-1}$ comes from the Delsarte, Goethals and Seidel's bound \cite{del77b} for the mimimum possible cardinality of spherical $\tau$-designs on $S^{n-1}$. The first few segments are as follows:
  $$
  [2,n),\; [n+1,2n) ,\;[2n,n(n+3)/2) ,\;[n(n+3)/2,n(n+1)) ,\;[n(n+1),n(n^2+6n+5)/6).
  $$
The results of \cite{BDHSS2016}
also imply the optimal choice of the polynomial, so the bounds we obtain cannot be improved by choosing a different polynomial of 
degree $\le m.$ The bound \eqref{ulb-general} will be expressed below in terms of $n$ and $N$ for $m=1$, 2, and 3.

Similarly, it is possible to bound the $h$-energy from above under the condition that the maximum inner product $s$ between distinct vectors in ${\cC}_N$ \blue{is fixed, or, allowing $n$ and $N$ to grow, satisfies the condition $\limsup_{n\to \infty} s<1.$ Note that if $n$ increases then so
does $N$, and the relation between them affects the asymptotic expressions.}
\blue{Consider the quantity
    $$ E_h(n,N,s):=\sup \left\{ E_h({\cC}_N): x \cdot y \leq s, x,y \in
 {\cC}_N, x \neq y\right\}, 
    $$
i.e., the supremum of $h$-energy of spherical codes of fixed dimension, cardinality, and minimum separation}. Universal upper bounds (UUBs) for $E_h(n,N,s)$
were derived in \cite{BDHSS2020}. To this end, the linear programming functional $f_0|\cC|-f(1)$ is minimized on the set of polynomials 
   $$
    \Big\{f(t)=\sum_{i=0}^{\deg(f)} f_i P_i^{(n)}(t): f(t) \geq h(t),t \in [-1,s]; f_i \leq 0,  i \geq 1 \Big\},
   $$
where $P_i^{(n)}(t)$ are the Gegenbauer polynomials (normalized by
$P_i^{(n)}(1)=1$). 
In \cite{BDHSS2020}, the authors use a specific choice of the polynomials $f(t)$ \blue{for fixed $n$, $N$, and $s$}
as explained in Section \ref{sec:proofs-bounds}. This leads to the bound \begin{equation} \label{uub-general}
\blue{E_h(n,N,s)} \leq \left(\frac{N}{N_1}-1\right)Nf(1)+N^2\sum_{i=0}^{k-1+\varepsilon} \rho_i h(\alpha_i),
\end{equation}
where this time the parameters $(\rho_i,\alpha_i)$ are functions of the dimension $n$ and the minimum separation $s,$  and \blue{$N_1=L_m(n,s)$, $m=2k-1+\varepsilon$, is the corresponding Levenshtein bound (see Sec.~\ref{sec:proofs-bounds} for additional details). The bound \eqref{uub-general} will be expressed below in terms of $n$, $N$, and $s$ for $m=1$, 2, and 3.}

\blue{While in this paper our focus is on codes of small size, a recent general result in \cite{BDHSS2019} (Theorem 7 and Corollary 1) implies the following asymptotic bound for the sum of distances:
\[     \tau_n({\cC}_N)\le \sqrt{2}N^2-\frac{N^{3/2}}{4\sqrt{2}}(1+o(1)), \]
which is applicable, in particular, for all $N$ such that $D^{\ast}(n,2l)\le N< D^{\ast}(n,2l+1), l\ge 1$.   }

\section{Bounds}\label{sec:bounds}

General bounds on energy of spherical codes obtained earlier in \cite{BDHSS2016} and \cite{BDHSS2020} apply to the sum of distances, although obtaining explicit expressions is not immediate. 
In this section we list the lower and upper bounds on the sum of distances obtained from the general results in the cited works, deferring the proof to Sec.~\ref{sec:proofs-bounds}. We limit ourselves to the first three bounds in the sequence of lower and upper bounds, noting that even in this case, the resulting expressions are unusually cumbersome.

\subsection{Upper bounds}
The following bounds on the maximum sum of distances of a spherical code in $n$ dimensions hold true:
   \begin{numcases}
{\tau(n,N) \leq} \tau_1(n,N):=N\sqrt{2N(N-1)}\label{eq:lb1}
\\[.1in]
\tau_2(n,N):=\frac{N\left(2N(N-n-1)+(N-2)\sqrt{2nN(n-1)(N-2)}\right)}{Nn+N-4n}  \label{eq:lb2}\\[.1in]
\tau_3(n,N) :=N \sqrt{\frac{2N(nA_1+2(N-n-1)^2B_1)}{n^2(n-1)^2+4n(N-n-1)(N-2n)}}  \label{eq:lb3}
  \end{numcases}
   where the first bound applies for $2\le N\le n+1,$ the second for $n+1\le N\le 2n,$ the third for $2n\le N\le n(n+3)/2,$
   and where
   \begin{gather} A_1= Nn^3+(2N-1)n^2-(N-1)(7N-2)n+(N-1)^2(2N+3), \label{eq:A1}\\
 B_1=\sqrt{n(n-1)N(N-n-1)}. \label{eq:B1}
   \end{gather}

Bound \eqref{eq:lb1} is attained by the simplex code, bound \eqref{eq:lb2} is attained by the biorthogonal
code, and bound \eqref{eq:lb3} is attained by all codes that meet the 3rd Levenshtein bound\footnote{\blue{All the known codes attaining Levenshtein bounds are listed in \cite[Table 9.1]{lev92}. There are two infinite series of codes as well as three sporadic examples that meet the 3rd bound. Some of these codes, originating from strongly regular graphs, were discovered in \cite{CGS1978} which established a condition for them to meet the 3rd Levenshtein bound; see \cite{lev92} for the details of this connection.}} \cite[p.620]{lev98}. 
Due to \eqref{eq:Stol}, these codes have the smallest quadratic discrepancy among all codes of their size. 

In the asymptotics of $n\to\infty$ bounds \eqref{eq:lb1} and \eqref{eq:lb2} yield
   \begin{align}
   \tau_1(n,N)&= \sqrt{2}N^2-\frac{N}{\blue{\sqrt{2}}}+O(1)   \quad\text{if $N\sim\delta n, 0<\delta\le 1$,} \label{eq:lb1a}\\
   \tau_2(n,N)&=\sqrt{2}N^2-2\Big(1-\delta-\frac{1-3\delta/2}{\sqrt 2}\Big) N+O(1) \quad \text{if $N\sim \delta n, 1\le \delta\le 2 $.} \label{eq:lb2a}
   \end{align}
Note that the bound \eqref{eq:lb2a} is slightly tighter than \eqref{eq:lb1a} because of a larger second term, which is greater
than $\frac 1{\sqrt{2}}$ for all $\delta > 1.$ The bound \eqref{eq:lb2a} is also uniformly better than \eqref{eq:disc-dist} for all $N=\delta n, \delta\in[1,2].$

The bound \eqref{eq:lb3} is valid for $N\le n(n+3)/2.$ Writing $N\sim \delta n^\alpha$, we note that 
its asymptotic behavior depends on $\alpha.$ For instance, for $N=\delta n^2$ we obtain
   \begin{equation} \label{eq:lb3a}
     \tau_3(n,N)= \sqrt 2 N^2 - \frac {\sqrt{2\delta}}{8} N^{3/2} + O(N).     
   \end{equation}
\blue{Here the order of the second term of the asymptotics coincides with the bound obtained from the average distance \eqref{eq:disc-dist} while the constant factor is better for all $
\delta>1.$}

\subsection{Lower bounds}
Let ${\cC}_N$ be a spherical code in $n$ dimensions, and assume that the minimum distance between distinct points $z_i,z_j\in {\cC}_N$ is bounded
from below, i.e., that $z_i\cdot z_j\le s$ for some $s\in[-1,1)$. Denote by $\tau_n(N,s)=\inf_{{\cC}_N}\tau_n({\cC}_N)$ the smallest possible sum of
distances for such codes. We have
    \begin{equation} \label{tau_lb}
    {\tau_n(N,s)\ge} \tau^{(i)}(n,N,s), i=1,2,3,
    \end{equation} 
where \blue{the bound}
    \begin{equation}\label{eq:ub1}
   \tau^{(1)}(n,N,s)=N(N-1)\sqrt{2(1-s)},
     \end{equation}
\blue{is applicable in \eqref{tau_lb} for $N \in [2,n+1]$ and $s \in [-1/(N-1),-1/n]$,}
the bound 
   \begin{equation}\label{eq:ub2}
   \tau^{(2)}(n,N,s)=\frac{N\left(2N(1-ns^2)-2n(1-s^2)+(n-1)\sqrt{2(1-s)}\right)}{n(1-s^2)},
   \end{equation}
is applicable for $n+1\le N \le 2n$ and $s \in \left[\frac{N-2n}{n(N-2)},0\right]$, and the bound
\begin{equation} \label{eq:ub3}
\tau^{(3)}(n,N,s)= \frac{N\left[A_5\left((1-s)(1+ns)A_4+B_4\sqrt{(1-s)B_5}\right)-2N(1+2s+ns^2)C_4\sqrt{A_6}\right]}{n(1-s)(1+2s+ns^2)^2C_4\sqrt{2B_5}},
\end{equation}
is applicable for $2n\le N \le n(n+3)/2$ and 
  \begin{equation}\label{eq:s}
  s \in \left[\frac{\sqrt{n^2(n-1)^2+4n(N-n-1)(N-2n)}-n(n-1)}{2n(N-n-1)},\frac{\sqrt{n+3}-1}{n+2}\right], 
  \end{equation}
where the notation in \eqref{eq:ub3} is as follows:   
  \begin{equation}\label{eq:AB}
  \begin{aligned}
 A_2 &= (1+ns)^5(1-s)+(n-1)^2((n+1)s+2), \\ 
B_2 &= (n-1)\sqrt{(1-s)(1+ns)((n+1)s+2)}\\
A_4 &= n(n+2)(n+3)s^4 + 2(3n^2+13n+8)s^3+2(n^2+12n+23)s^2+2(2n^2+5n+17)s+9n+3, \\
 B_4 &= 2(n-1)((n+1)s+2)((n-2)s^2-2ns-1), \\
 C_4 &= 2n(n+2)s^3-(n^2-5n-2)s^2-6ns-n-5, \\
 A_5 &= N(1-ns^2)-n(1-s)((n+1)s+2), \\
B_5 &=\frac{(n+1)s+2}{1+ns}, \\
A_6 &= \frac{(1-s)(A_2+2(1+ns)^2B_2)}{1+ns}.
 \end{aligned}
\end{equation}
  
{\em Remarks.}

1. Note that expression \eqref{eq:ub1} yields a trivial bound on the sum of distances, assuming that every pair of code points is at distance $\sqrt{2(1-s)}.$ It is included for completeness because it follows by optimizing the linear polynomial in the linear programming problem.

2. The bounds \eqref{eq:ub1}--\eqref{eq:ub3} are proved for $s$ in the specified intervals above but are valid \blue{at least for slightly 
larger $s$ (by continuity). For example, the bound \eqref{eq:ub1} is 
valid for all $s$. The lower limits for $s$ are determined from the
inequality $N_1 \geq N$ and the upper limits are the same as for the Levenshtein bound $L_m(n,s)$ (see in Sec.~\ref{sec:proofs-bounds} for more details).} 

3. Using Mathematica, we can compute asymptotic behavior of $\tau^{(3)}(n,N,s)$ for $n\to\infty.$ Since it depends on $s$, we do not include general expressions, leaving this for the examples.

\section{Examples of codes of small size}
In this section we consider several families of spherical codes that attain the asymptotic 
extremum of the sum of distances. We focus on sets with a small number of distinct distances because the sum of distances is easier to compute, and because their cardinalities fit the range of the parameters used to derive the bounds in the previous section.
We consider three types of objects, families of equiangular lines, strongly regular graphs, and binary codes. General introductions to their properties are found in \cite[Ch.11]{GodsilRoyle2001}, \cite{Brouwer2022}, and \cite{mac91}, respectively.

\subsection{Equiangular lines}
A family of $M$ equiangular lines in $\reals^n$ with common inner product $s$ defines a spherical code 
${\cC}_N$ with $N=2M$ vectors, each of which has inner product $s$ with $M-1$ other vectors and $-s$ with their opposites. The sum
of distances in ${\cC}_N$ equals
    \begin{align}\label{eq:sum}
    \tau_n({\cC}_N)&=\sum_{i,j=1}^{N}\|z_i-z_j\|=N((M-1)(\sqrt{2-2s}+\sqrt{2+2s})+2)\nonumber
    \\&=\frac{N^2}{\sqrt{2}}(\sqrt{1-s}+\sqrt{1+s})+O(N).
    \end{align}

For small $s$ we can write $ \sqrt{1-s}+\sqrt{1+s}= 2-\frac{s^2}4+O(s^4),$ so \blue{for $M=\Theta(n^2)$} the sum of distances will be close to the value $\sqrt2 N^2$ given by the bound \eqref{eq:lb3a}. \blue{Example 1 below illustrates this claim.}

{\sc Examples.} 

1. Constructions with $M=\Theta(n^2).$ There are several constructions of large-size sets of equiangular lines,
starting with De Caen's family \cite{deCaen2000}; see also \cite{Jedwab2015}. In all these constructions $s \to 0,$ and thus the sum of distances equals $\tau_n({\cC}_N)=\sqrt 2N^2(1+o(1)),$ showing that such families yield asymptotically optimal spherical codes. For instance, De Caen's family yields codes ${\cC}_{N}$ with the parameters 
$$n=3\cdot 2^{2r-1}-1, \ N=\frac 49 (n+1)^2, \ s=\frac 1{2^r+1}, \quad r\ge 1,$$
 and we find from \eqref{eq:sum} that 
 \[ \tau_n({\cC}_N)=\sqrt 2 N^2-\frac{1}{4\sqrt 2}N^{3/2}+O(N^{5/4}).\] 
 At the same time, on account \eqref{eq:lb3a} and \eqref{eq:ub3} any
 sequence of codes ${\cC}_N$ with $N\sim\frac49 n^2$ and $s\sim\sqrt{\frac3{2n}}$ satisfies
    $$
       \sqrt 2 N^2 - \frac{1}{5\sqrt 2} N^{7/4}-O(N^{3/2})  \le \tau_n({\cC}_N) \le \sqrt 2 N^2 - \frac 1{6\sqrt 2} N^{3/2}+O(N)
    $$
(computations for the lower bound performed with Mathematica).  
We give examples of the bounds on the sum of distances of de Caen's codes and of its true value for the first few values of $r.$
{\small
\begin{table}[H]
\begin{tabular}{|c|c|c|c|c|c|c|}
  \hline
  $r$ & $n$ & $N$ & Upper bound $\tau_3(n,N)$  & $\tau_n({\cC}_N)$ & Lower bound $\tau^{(3)}(n,N,s)$   \\[0.03in]
 \hline
    3 &    95 &      4096 &   $2.369344 \cdot 10^7$ &    $2.368643\cdot 10^7$ &    $2.341901 \cdot 10^7$  \\[0.02in]
 \hline
    4 &   383 &     65536 &    $6.0719880 \cdot 10^9$ &    $6.071317 \cdot 10^9$ &    $6.036098 \cdot 10^9$  \\[0.02in]
 \hline
    5 &  1535 &   1048576 & $1.5548171 \cdot 10^{12}$ & $1.554765 \cdot 10^{12}$ & $1.550113 \cdot 10^{12}$  \\[0.02in]
 \hline
    6 &  6143 &  16777216 & $3.9805762.10^{14}$ & $3.980539 \cdot 10^{14}$ & $3.974463 \cdot 10^{14}$  \\[0.02in]
 \hline
    7 & 24575 & 268435456 & $1.0190430 \cdot 10^{17}$ & $1.019041 \cdot 10^{17}$ & $1.018254 \cdot 10^{17}$  \\
 \hline
\end{tabular}
\end{table}
}

2. Below by $M_s(n)$ we denote the maximum number of equiangular lines in $n$ dimensions with inner product $s$. It is known \cite{LemmensSeidel1973} that $M_{1/3}(n)=2(n-1).$ Taking $N=4(n-1)$ for a given $n,$ we obtain a spherical code ${\cC}_N$ with sum of
distances equal to
   $$
   \tau_n({\cC}_N)=N((M-1)(\sqrt{4/3}+\sqrt{8/3})+2)=N^2\frac{1+\sqrt 2}{\sqrt{3}}(1+o(1)).
   $$
The constant factor in this expression is approximately $1.39.$ 
A more detailed calculation shows that 
  $$
  \lim_{n\to\infty}\frac{\tau_n({\cC}_N)}{\tau_3(n,N)}=(\sqrt 6(\sqrt 2-1))^{-1}\approx 0.9856.
  $$

3. Further, by \cite{Neumaier1989},  $M_{1/5}(n)=\lfloor 3(n-1)/2\rfloor$ for all sufficiently large $n$. This set of lines
yields a spherical code with sum of distances
  $
  \tau_n({\cC}_N)=N^2((\sqrt{2}+\sqrt{3})/\sqrt{5})(1+o(1))\approx 1.40\blue{7} N^2,$ which is again very close to \eqref{eq:lb3a}.
It is not difficult to check that  
\[   \lim_{n\to\infty} \frac{\tau_n({\cC}_N)}{\tau_3(n,N)}= \left(\sqrt2+\sqrt3\right)/\sqrt{10}\approx 0.9949.\]
  
4. A recent paper by Jiang and Polyanskii \cite{Jiang2020a} shows that $M_{1/(1+2\sqrt 2)}(n)=3n/2+O(1),$ yielding a spherical
code of size $N=3n+O(1).$ For this code, the constant factor in \eqref{eq:sum} equals 
     $$
     \frac 1{\sqrt2}\Big(\sqrt{1-\frac{1}{1+2\sqrt 2}}+\sqrt{1+\frac{1}{1+2\sqrt 2}}\Big)\approx 1.40189.
     $$
In the limit of $n\to\infty$, the sum of distances \blue{satisfies} $\tau_n({\cC}_N)/\tau_3(n,N)\to 0.991.$

More examples can be generated relying on constructions of equiangular line sets of size
$O(n^{3/2})$ based on Taylor graphs and projective planes \cite{LemmensSeidel1973}. Recent additions to the literature include new upper bounds and exact asymptotics of the size of equiangular line sets with fixed inner product $s$ \cite{Balla2018,GlazyrinYu2018,Jiang2021}.

\subsection{Strongly regular graphs and tight frames}
Here we consider the sum-of-distances function for spherical codes obtained from strongly regular graphs (SRG). A $k$-regular graph on $v$ vertices is strongly
regular if every pair of adjacent vertices has $a$ common neighbors and every pair of nonadjacent vertices has $c$ common neighbors. Below we use the notation $\SRG(v,k,a,c)$ when we need to mention the parameters explicitly. 

\blue{The spectral structure of SRGs is well known; see for instance \cite[p.~118]{Brouwer2012}, \cite{CGS1978}, or
\cite[Sec.~9.4]{Ericson2001} (the last two references highlight the relation between spherical codes and SRGs and more generally, association schemes).} The adjacency matrix of an SRG has three eigenspaces that correspond to the eigenvalues $k,r_1,r_2.$ 
Let $\Delta=(a-c)^2+4(k-c)$, then the eigenvalues other than $k$ have the form
   $$
  r_1=\frac 12(a-c+\sqrt\Delta),\quad r_2=\frac 12(a-c-\sqrt\Delta),
  $$
and the dimensions of the corresponding eigenspaces are
  \begin{equation}\label{eq:dimensions}
  n_{1,2}=\frac 12\Big(v-1\pm\frac{(v-1)(c-a)-2k}{\sqrt\Delta}\Big),
  \end{equation}
\blue{where we write $n_{1,2}$ to refer to both eigenspaces at the same time.} 
 
Spherical embeddings of SRGs were introduced by Delsarte, Goethals, and Seidel \cite{del77b}, Example~9.1. \blue{To obtain a spherical code from an SRG, assign vectors of the standard basis of 
$\reals^v$ to the vertices, and then project the basis on an eigenspace of the graph.}
In particular, using the eigenspace $W_{r_1}$ that corresponds to $r_1,$ we obtain a spherical code in $\reals^{n_1}$ with $N=v$ points and inner products
  \begin{equation}\label{eq:stheta}
  s_1=\frac{r_1}{k},\quad s_2=-\frac{1+r_1}{v-1-k}.
  \end{equation}
A similar procedure for $r_2$ yields a spherical code in $\reals^{n_2}$ with $v$ points and inner products
  \begin{equation}\label{eq:stau}
  s_1=-\frac{1+r_2}{v-1-k},\quad s_2 =\frac {r_2}k,
  \end{equation}
where in both cases $s_1\ge 0>s_2.$  We again reference \cite[Sec.~9.4]{Ericson2001} for the 
details and \cite{BB2005} for a short proof.

The distribution of distances in the obtained spherical codes does not depend on the point $z_i\in {\cC}_N.$ If
the code is obtained by projecting on $W_{r_1}$, then the number of neighbors of a point with inner product $r_1/k$ is $k$, and if it is obtained by projecting on $W_{r_2},$ then the number of neighbors of a point with inner product $r_2/k$ is $k$. Thus, in both cases, the number of neighbors with the remaining value of the inner product is $N-k-1.$

Combining \eqref{eq:dimensions}, \eqref{eq:stheta}, and \eqref{eq:stau}, we obtain 
\begin{proposition} Projecting an $\SRG(v,k,a,c)$ on the eigenspace $W_{\theta}, \theta=r_1,r_2$ results in a spherical code in $\reals^{n_{1,2}}$ of size $N=v$ whose sum of distances equals
   \begin{equation}\label{eq:stf}
   \tau_{n_{_{1,2}}}({\cC}_N)=N\left(\sqrt{2k(k-\theta)}+\sqrt{2(N-1-k)(N+\theta-k)}\right),
   \end{equation}
   where $\theta=r_1$ or $r_2$ as appropriate. 
\end{proposition}

{\em Remark:} Families of spherical  codes considered below attain sums of distances that can be written in the form
$\tau_n({\cC}_N)=\sqrt 2 N^2(1+o(1)).$ A sufficient condition for this is that the eigenvalues are small compared to $N$, as can be seen upon rewriting \eqref{eq:stf} in the form
     $$
     \tau_{n_i}({\cC}_N)=\sqrt 2 N^2\Big(\sqrt{\frac{k(k-\theta)}{N^2}}+\sqrt{\frac{(N-k-1)(N-k+\theta)}{N^2}}\;\Big).
     $$
As long as $\theta/N=o(N),$ as is the case in the examples below, the main term of the asymptotic expression will be $\sqrt2 N^2.$

Spherical codes obtained from SRGs have an additional property of forming tight frames for $\reals^{n_{1}}$ or $\reals^{n_2}.$ Recall that a spherical code ${\cC}_N=\{z_1,\dots,z_N\}$ forms a {\em tight frame} for $\reals^d$ if $\sum_{i=1}^N (x \cdot z_i)^2=A\|x\|^2$ for any $x\in R^n$, where $A$ is a constant. A necessary and sufficient condition for the tight frame property to hold is the equality 
\cite{Benedetto2003}
    \begin{equation}\label{eq:fp}
    \sum_{i,j=1}^N (z_i\cdot z_j)^2=\frac {N^2}{n}.
    \end{equation}
In the frame theory literature the sum on the left-hand side of \eqref{eq:fp} is called the {\em frame potential} \cite{Waldron2018}. 

It turns out that all two-distance tight frames are obtained as spherical embeddings of SRGs \cite{BGOY15,Waldron2009}.

\noindent{\sc Examples.}

The families of graphs considered below are taken from the online database \cite{Brouwer}.

 1. Graph of points on a quadric in PG$(2m,q).$ The parameters of the SRG are 
   $$
   v=\frac{q^{2m}-1}{q-1},\;\; k=\frac{q(q^{2m-2}-1)}{q-1},\;\;a=\frac{q^2(q^{2m-4}-1)}{q-1}+q-1, \;\;c=\frac{q^{2m-2}-1}{q-1},
   $$
and the eigenvalues are $r_{1,2}=\pm q^{m-1}-1.$
Spherical embeddings of this graph give tight frames in dimensions \eqref{eq:dimensions}
   $$
   n_{1,2}=\frac12(N-1\pm q^m)\approx\frac12(N\pm\sqrt N),
   $$
which is easily seen since $\sqrt\Delta=2q^{m-1}.$
The size of the code ${\cC}_N={\cC}_N(r_1)$ is $N=v$ and the sum of distances is computed from \eqref{eq:stf} and equals
   $$
   \tau_{n_1}({\cC}_N)=N\sqrt{2(q^m+1)}\Big[\frac{q^{m-1}-1}{q-1}\sqrt{q(q^{m-1}+1)}+q^{\frac{3m-2}2}\Big].
   $$
   Taking $m\to\infty,$ we compute
   \begin{equation}\label{eq:second}
   \tau_{n_1}({\cC}_N)=\sqrt2 N^2-\frac{5}{4\sqrt2}N+O(1).
   \end{equation}
Since in this case $N\approx2n_1-2\sqrt {2n_1},$ the appropriate bound to look at is $\tau_2(n,N)$ with $\delta=2.$ The second term of the sum of distances in \eqref{eq:second} is approximately $-0.884 N$ while the second term in \eqref{eq:lb2a} is $-2(\sqrt 2-1)N\approx -0.828N.$
 
Likewise, the projection on the eigenspace $W_{r_2}$ gives a spherical code ${\cC}_N={\cC}_N(r_2)$ whose sum of distances equals
   $$
   \tau_{n_2}({\cC}_N)=N\sqrt{2(q^m-1)}\Big[\frac{q^{m-1}+1}{q-1}\sqrt{q(q^{m-1}-1)}+q^{\frac{3m-2}2}\Big].
   $$
For large $m$ this behaves as $\sqrt 2 N^2-\frac{5}{4\sqrt2}N+O(1),$ exhibiting similar behavior as the code 
in dimension $n_1.$

2. Graph of points on a hyperbolic quadric in PG$(2m-1,q)$. The parameters of the SRG are 
  \begin{equation}\label{eq:he}
  v=\frac{q^{2m-1}-1}{q-1}+q^{m-1},\;\; k=\frac{q(q^{2m-3}-1)}{q-1}+q^{m-1},\;\; a=k-q^{2m-3}-1,\;\; c=k/q,
  \end{equation}
and the eigenvalues are $r_{1}=q^{m-1}-1$ and $r_2=-q^{m-2}-1.$ 
Using \eqref{eq:he}, we obtain that
the dimensions of the spherical embeddings of this graph are
   $$
   n_1=\frac{q(q^{m-2}+1)(q^m-1)}{q^2-1}, \quad n_2=\frac{q^2(q^{2m-2}-1)}{q^2-1}
   $$
and thus, $n_1\approx N/(q+1), n_2\approx Nq/(q+1).$ The sum of distances in ${\cC}_N(r_1)$ is found to be
   $$
   \tau_{n_1}({\cC}_N)=N\sqrt{2q(q^{m-1}+1)}\Big[\frac{q^{m-1}-1}{q-1}\sqrt{q^{m-2}+1}+q^{\frac{3m}2-2}\Big].
   $$
For large $m$ we obtain $\tau_{n_1}({\cC}_N)=\sqrt 2 N^2-\frac{q+4}{4\sqrt 2}N-O(1).$ At the same time, from the bound
\eqref{eq:lb3} we obtain an upper estimate of the form $\sqrt 2 N^2-O(N),$ giving the second term of the same order, although with a smaller constant factor.

Turning to the code ${\cC}_N$ obtained by projecting on the eigenspace $W_{r_2},$ we find that
    $$
    \tau_{n_2}({\cC}_N)=N\sqrt{2(q^m-1)}\Big[\frac{q^{m-2}+1}{q-1}\sqrt{q(q^{m-1}-1)}+q^{\frac{3m}2-2}\Big],
    $$
yielding $\tau_{n_2}({\cC}_N)=\sqrt 2N^2-\frac{4q-1}{4q\sqrt 2}N-O(1)$, with similar conclusions in regards to asymptotics of the upper bound.

{\em Remark:}  It is known \cite{Benedetto2003} that $N^2/n$ is the smallest value of the frame potential in  over all $(n,N)$ spherical codes. Thus, two-distance tight frames form spherical codes in $R^n$ that have asymptotically maximum sum of distances while also minimizing the frame potential.

\subsection{Spherical embeddings of binary codes} Infinite sequences of asymptotically optimal spherical codes can be obtained by spherical embeddings of binary codes. 
Let ${\cC}_N\subset \cX_n = \{0,1\}^n$ be a binary code of length $n$ and 
cardinality $N$, and denote by $A_w=\frac{1}{N}\#\{a,b\in C: d_H(a,b)=w\}$ the average number of neighbors of a code vector at Hamming distance $w$. The $(n+1)$-tuple $(A_0=1,A_1,\dots,A_n)$ is called
the distance distribution of the code ${\cC}_N$. For a vector $z\in\cX_n$ denote by $\tilde z$ the $n$-dimensional real vector given by $\tilde z_i=(-1)^{z_i}/\sqrt n, i=1,\dots,n,$ and let $\tilde {\cC}_N\subset S^{n-1}$ be the spherical embedding of the code ${\cC}_N.$
Since $\|\tilde x-\tilde y\|=2\sqrt{d_H(x,y)/n},$ the sum
of distances in $\tilde {\cC}_N$ can be written as
    \begin{equation}\label{eq:2s}
    \sum_{i,j=1}^N\|\tilde z_i-\tilde z_j\|=\frac{2N}{\sqrt n}\sum_{w=0}^n A_w\sqrt w.
    \end{equation}
Using this correspondence, we give several examples of asymptotically optimal families of spherical codes.

\subsubsection{Sidelnikov codes} In \cite[Thm.~7]{Sid1971}, Sidelnikov constructed a class of binary linear codes $C_r, r\ge1$ 
with the parameters
   $
    n=\frac{2^{4r}-1}{2^r+1}, \;N=2^{4r}. 
    $
The distance distribution of the codes has two nonzero components (in addition to $A_0=1$):
     \begin{equation*}
   \begin{aligned}
   &w_1=\frac{2^{4r-1}-2^{2r-1}}{2^r+1}, \;A_{w_1}=2^{4r}-n-1,\\
   &w_2=\frac{2^{4r-1}+2^{3r-1}}{2^r+1},\;A_{w_2}=n.   
   \end{aligned}
   \end{equation*}
   
Let us compute the sum of distances of the spherically embedded Sidelnikov codes. Using \eqref{eq:2s}, we obtain
  $$ 
\frac {2N}{\sqrt n} (A_{w_1}\sqrt{w_1}+A_{w_2}\sqrt{w_2}))=\sqrt2 \Big(N^2 - \frac{1}{8}N^{5/4}-\frac{7}{16}N-\frac{13}{128} N^{3/4}\Big)+O(N^{1/2}).
  $$
At the same time, the bounds \eqref{eq:lb3a} and \eqref{eq:lb3} imply that for any sequence of codes ${\cC}_N$ with $N$ as above and
$s=1-2w_1/n$
    $$
          \sqrt 2 N^{2}-\frac{1}{2\sqrt 2}N^{7/4}-O(N^{11/8}) \le \tau_n({\cC}_N)\le \sqrt2 \Big(N^2 - \frac{1}{8}N^{5/4}-
          \frac{7}{16}N-\frac{5}{128}N^{3/4}\Big)+O(N^{1/2}),
   $$
   and so as $r\to \infty$ the true value agrees with the upper bound in the first three terms.
The first few values of the sum of distances together with the bounds of Sec.~\ref{sec:bounds} are shown in the table below.

{\small
\begin{table}[H]
\begin{tabular}{|c|c|c|c|c|c|c|}
  \hline
  $r$ & $n$ & $N$ & Upper bound $\tau_3(n,N)$  & $\tau_n({\cC}_N)$ & Lower bound $\tau^{(3)}(n,N,s)$   \\[0.03in]
 \hline
  1 &     5 &      16 &         $345.4941208$ &         $345.4941208$ &         $345.4941208$   \\[.02in]
\hline
    2 &    51 &     256 &       $92338.0198$ &       $92334.5230$ &       $91959.9016$   \\[.02in]
\hline
    3 &   455 &    4096 &      $2.371820900\cdot 10^7$ &      $2.371817158 \cdot 10^7$ &      $2.369984979 \cdot 10^7$   \\[.02in]
\hline
    4 &  3855 &   65536 &      $6.0737748 \cdot 10^9$ &      $6.0737745 \cdot 10^9$ &      $6.073097678 \cdot 10^9$   \\[.02in]
\hline
    5 & 31775 & 1048576 & $1.554937673 \cdot 10^{12}$ & $1.554937671 \cdot 10^{12}$ & $1.554914842\cdot 10^{12}$   \\
  \noalign{\smallskip}\hline
\end{tabular}
\end{table}
}
\vspace*{-.2in}The relative difference between the upper bound and the true value for $r=5$ is about $10^{-9},$ and the upper and lower bounds on the sum of distances are also rather close.

We next discuss some families of spherical codes obtained from binary codes of cardinality $N\approx n^2$ that share the following common property: they have a small number of nonzero distances concentrated around $n/2.$ Since the factor $\sqrt w\approx \sqrt\frac{n}2$ 
for large $n$ can be taken outside the sum in \eqref{eq:2s}, and since the nonzero coefficients $A_w$ add to $N-1$, all such families satisfy
   $$
   \tau_n({\cC}_N)\sim \sqrt 2 N^2(1+o(1)),
   $$
differing only in the lower terms of the asymptotics.

\subsubsection{Kerdock codes} \cite[\S15.5]{mac91}. Binary Kerdock codes form a family of nonlinear codes of length $n=2^{2m}, m\ge 2$ and cardinality
$N=n^2.$ The distribution of Hamming distances does not depend on the code point and the nonzero entries $(A_i)$ are as follows:
   $$
   A_0=A_n=1, A_{(n\pm\sqrt n)/2}=n(n/2-1), A_{n/2}=2(n-1).
   $$
From \eqref{eq:2s}, the sum of distances of the spherical Kerdock code equals
   $$
   \tau_{n}(\tilde {\cC}_N)=\sqrt2 N^2 -\frac{1}{4\sqrt{2}} N^{3/2} +O(N),
   $$
which agrees with the bound \eqref{eq:disc-dist}, \eqref{eq:lb3a}. Note that    
for general completely monotone potentials, the first-term optimality of the Kerdock codes was previously observed in \cite{BDHSS2015}.
 
\subsubsection{Dual BCH codes} \cite[\S15.4]{mac91}. Let ${\cC}_N$ be a linear binary BCH code of length $n=2^r-1$, $r\ge 3$ with \blue{minimum} distance $5$. 
Suppose that $r$ is odd. Then the dual code $({\cC}_N)^\bot$ has cardinality $N=2^{2r}$ and distance distribution $A_0=1$ and
  $$
  A_{\frac{n+1}2\pm \sqrt{\frac{n+1}2}}=n\Big(\frac {n+1}4\mp\frac{\sqrt{n+1}}{2\sqrt 2}\Big), \; A_{\frac{n+1}2}=\frac{n(n+3)}2.
  $$
\blue{For $r$ even the dual BCH code of length $2^r-1$ has distance distribution $A_0=1$ and
  \begin{gather*}
  A_{\frac{n+1}2\mp\sqrt{n+1}}=\frac12 n\sqrt{n+1}\Big(\sqrt{\frac{n+1}4}\pm1\Big),
  \;A_{\frac{n+1}2\mp\sqrt{\frac{n+1}2}}=\frac13n\sqrt{n+1}(\sqrt{n+1}\pm1)\\
  A_{\frac{n+1}2}=n\Big(\frac{n+1}4+1\Big).
  \end{gather*}
  %(\cite[p.452]{mac91}).
}
Using \eqref{eq:2s}, we find that the sum of distances in both cases comes out to be 
\[ \tau_{n}((\tilde {\cC}_N)^\bot)= 
\sqrt 2N^2- \frac{1}{4\sqrt{2}}N^{3/2}-O(N). \] Note that $\tau_n((\tilde {\cC}_N)^\bot)$ follows closely the upper bound \eqref{eq:disc-dist}.

Many more similar examples can be given using the known results on binary codes with few weights \cite[Ch.15]{mac91}, \cite{CK1986,D2016,LLHQ2021,WZHZ2015} (this list is far from
being complete). At the same time, obviously there are sequences of binary codes $({\cC}_N)$ 
\blue{that yield spherical codes whose sum of distances differs significantly from $\sqrt 2 N^2$}. For instance, consider the code ${\cC}_N$ formed of $\binom n2$ vectors of Hamming weight 2, then the pairwise distances are 2 and 4, and a calculation shows that $\tau_N(\tilde {\cC}_N)= (2N)^{7/4}(1+o(1)).$

\section{Sum of distances and bounds for quadratic discrepancy of binary codes}\label{sec:binary}
An analog of Stolarsky's identity \eqref{eq:Stol} for the Hamming space $\cX_n=\{0,1\}^n$ was recently derived in \cite{Barg2021}. 
For a binary code ${\cC}_N\in\cX_n$ define the {\em quadratic discrepancy} as follows:
   \begin{equation*}
   D_b^{L_2}({\cC}_N)=\sum_{t=0}^n\sum_{x\in \cX}\Big(\frac{|B(x,t)\cap {\cC}_N|}{N}-\frac{v(t)}{2^n}\Big)^2,
   \end{equation*}
where $B(x,t)=\{y\in \cX_n: d_H(x,y)\le t\}$ is the Hamming ball centered at $x$ and $v(t)=\sum_{i=0}^t\binom ni$ is its volume. 
Note that we again abuse the terminology since strictly speaking, $D_b^{L_2}({\cC}_N)$ is a square of the discrepancy; see
also footnote 1 above. We use the subscript $b$ to differentiate this quantity from it spherical counterpart defined in \eqref{eq:disc}.
\blue{To state the Hamming space version of Stolarsky's identity, let us define a function
$\lambda:{\mathbb Z}\to{\mathbb Z}$. By definition, $\lambda(0)=0$ and for $w=2i,1\le i\le \lfloor n/2\rfloor$}
   \begin{equation}\label{eq:lambda}
   \lambda(w-1)=\lambda(w)=2^{n-w}i\binom{w}{i}.
   \end{equation} 
An analog of relation \eqref{eq:disc} in the binary case has the following form:
  \begin{equation*}
    D_b^{L_2}({\cC}_N)=\frac n{2^{n+1}}{\binom{2n}n}-\frac1{N^2}\sum_{i,j=1}^N\lambda(d_H(z_i,z_j)).
  \end{equation*}
The average value of $\lambda(\cdot)$ over the code can be written in the form
  \begin{equation}\label{eq:dd}
   \frac1{N^2}\sum_{i,j=1}^N\lambda(d_H(z_i,z_j))=\frac1{N}\sum_{w=1}^n A_w\lambda(w),
   \end{equation}
where $(A_w,w=1,\dots,n)$ is the distribution of distances in ${\cC}_N$ defined above. 
Thus, the value of discrepancy of the code is determined once we know the average ``energy'' for the potential $\lambda,$ denoted
$\langle \lambda\rangle_{{\cC}_N}$. Some estimates of this quantity were proved in \cite{Barg2021}.

In this section we note that the bounds on the sum of distances derived above in Sec.~\ref{sec:bounds} imply bounds on 
$\langle \lambda\rangle_{{\cC}_N}$ via the spherical embedding, and thus also imply bounds on $D_b^{L_2}.$
Our results are based on the following simple observation. 

\begin{proposition}\label{prop:bs} Let $n$ be even and let ${\cC}_N\subset\cX_n$ be a binary code and let $\tilde {\cC}_N\subset S^{n-1}$ be its spherical embedding. We have
  \begin{equation}\label{eq:bs}
  \langle \lambda\rangle_{{\cC}_N}\le \frac{2^{n-1}}{N^2}\sqrt{\frac{n}{\pi}}\tau_n(\tilde {\cC}_N)
  \end{equation}
\end{proposition}
\begin{proof} Assume that $n$ is even. From \eqref{eq:dd} and \eqref{eq:lambda} we obtain
   \begin{align*}
   \frac 1N\sum_{i,j=1}^N \lambda(d_H(z_i,z_j))&=\sum_{w=1} A_w\lambda(w) \leq \sum_{i=1}^{n/2}(A_{2i-1}+A_{2i}) 2^{n}\sqrt {i/\pi}\\&= \frac{2^{n-1/2}}{\sqrt{\pi}}\sum_{i=1}^{n/2}(A_{2i-1}+A_{2i})\sqrt{2i}\le \frac{2^{n}}{\sqrt{\pi}}\sum_{w=1}^n A_w\sqrt w
   \end{align*}
where for the first inequality we used the estimate $i\binom{2i}i\le \sqrt{i/\pi}\,2^{2i},$ valid for all $i.$ 
Substituting the value of the sum from \eqref{eq:2s}, we obtain the claim.
\end{proof}
With minor differences, this result is also valid for odd $n$. 

Earlier results \cite[Thm.5.2]{Barg2021} give several estimates for average value of $\lambda;$ for instance,
for $n=2l-1$, $l$ even
   \begin{equation*}\label{eq:lambda1}
   \langle \lambda\rangle_{{\cC}_N} \le\lambda(l)(1-\frac1{2N}).
    \end{equation*}
Using this inequality and estimating the binomial coefficient, we obtain
  \begin{equation}\label{eq:upper-lp}
  \langle \lambda\rangle_{{\cC}_N}\le  2^{n-l} \frac l2\binom{l}{l/2}\le 2^{n-1/2} \sqrt{ l/{\pi}},
  \end{equation}
valid for all odd $n$. 
While in \cite{Barg2021} inequality \eqref{eq:lambda} is proved by linear programming in the Hamming space, similar estimates are also obtained 
from \eqref{eq:bs} and the upper bounds \eqref{eq:lb1}-\eqref{eq:lb3} (for $N$ in the range of their applicability), and they largely 
coincide with earlier results. For instance, using \eqref{eq:bs} and a bound of the form \eqref{eq:lb3a} with $N=\delta n^2,$ 
we obtain $\langle \lambda\rangle_{{\cC}_N}\le 2^{n-\frac12}\sqrt{\frac n\pi}(1-O(N^{-1/2})),$ which is only slightly inferior to  \eqref{eq:upper-lp}.

In summary, spherical embeddings of binary codes give an alternative way of proving lower bounds for their quadratic discrepancy.

\section{Proofs of the bounds}\label{sec:proofs-bounds}

In this section, we prove the bounds on the sum of distances stated in Sec.~\ref{sec:bounds}, using the energy function $E_h(n,N)$ with 
$h(t)=L(t)=-\sqrt{2(1-t)}$ (the negative distance). Accordingly, the upper and lower bounds of Sec.~\ref{sec:bounds} exchange their roles. All the derivatives $L^{(i)}(t)$, $i \geq 1$, are defined and positive in $[-1,1)$ and $\lim_{t \to 1^-} L^{(i)}(t)=+\infty$; $L(t)+2$ is nonnegative and increasing in $[-1,1]$, and thus $L(t)$ is absolutely monotone up to an additive constant. \blue{Thus, $L(t)$ fits the frameworks for ULB and UUB from \cite{BDHSS2016} and \cite{BDHSS2020}, respectively (the possible ULB application was mentioned already in the introduction of \cite{BDHSS2016}). }

\subsection{Derivation of the necessary parameters}

Here we explain the choice of the parameters in the Levenshtein framework used to derive the bounds.

The parameters $k$, $\varepsilon$, $m=2k-1+\varepsilon$, and $(\rho_i,\alpha_i)$, $i=0,1,\ldots,k-1+\varepsilon$, 
originate in the paper of Levenshtein \cite{lev92} (see also \cite[Section 5]{lev98}), where the author used them 
to establish optimality of his bound on the size of codes (see Theorem 5.39 in \cite{lev98}). 

For each positive integer $m=2k-1+\varepsilon$, where $\varepsilon \in \{0,1\}$ accounts for the parity of $m$, Levenshtein
used the degree $m$ polynomial
    $$
     f_m^{(n,s)}(t)=(t-\alpha_0)^{2-\varepsilon}(t-\alpha_{k-1+\varepsilon}) \prod_{i=1}^{k-2+\varepsilon} (t-\alpha_i)^2 
     $$
to obtain his universal upper bound $L_m(n,s)$ on the maximal cardinality of a code on $S^{n-1}$ with separation $s$. 
The numbers $\alpha_0 < \alpha_1 < \cdots < \alpha_{k-1+\varepsilon}$ belong to $[-1,1)$ and $\alpha_{k-1+\varepsilon}=s$ and 
$\alpha_0=-1$ if and only if $\varepsilon=1$. The polynomial $f_m$ can be written in the form
     \begin{equation} \label{lev-poly}
 f_m^{(n,s)}(t)=(t+1)^{\varepsilon} \left(P_k(t)P_{k-1}(s) - P_k(s)P_{k-1}(t)\right)^2/(t-s),
     \end{equation}
where $P_i(t)=P_i^{(\frac{n-1}{2},\frac{n-3}{2}+\varepsilon)}(t)$ is the Jacobi polynomial normalized to satisfy $P_i(1)=1$.
For small $m$ the zeros $\alpha_i$ of $f_m$ can be easily found.

The quadrature formula
\begin{equation} \label{QF}
f_0=\frac{f(1)}{L_m(n,s)}+\sum_{i=0}^{k-1+\varepsilon} \rho_i f(\alpha_i), 
\end{equation}
which is exact for all real polynomials $f(t)=\sum_{i=0}^{d} f_i P_i^{(n)}(t)$ of degree $d \leq m$, reveals a strong relation
between the Levenshtein bounds and the energy bounds, \blue{as explained in the next paragraph (for more details, see \cite[Section 2.2]{BDHSS2016} and \cite[Section 3.1]{BDHSS2020}). We also use \eqref{QF}} to calculate  the weights $\rho_i$; see, for example, \cite{BDL1999}, where the formulas for $\rho_i$ for odd $m$ were derived from a Vandermonde-type system. \blue{We also note that $L_m(n,s)=f_m^{(n,s)}(1)/f_0$, where $f_0$ is the constant coefficient of $f_m^{(n,s)}$.}

Formula \eqref{QF} is instrumental in the representation \eqref{ulb-general} of the ULB for the energy $E_h({\cC}_N)$ and the proof of its optimality in \cite{BDHSS2016}. \blue{For ULB, we need polynomials that are positive definite (i.e., their Gegenbauer
expansions have nonnegative coefficients) and such that $f \leq h$ in $[-1,1]$. First, $m=2k-1+\varepsilon$ is determined by the rule
$N \in [D^*(n,m),D^*(n,m+1)$. Hermite interpolation with $f(\alpha_i)=h(\alpha_i)$, where the nodes $\alpha_i$, $i=0,1,\ldots,k-1+\varepsilon$ arise as the roots of $L_m(n,s)=N$ considered as an equation in $s$, provides an LP polynomial satisfying
both requirements \cite[Theorem 3.1]{BDHSS2016}. } Then 
the quantity $f_0N-f(1)$, which gives rise to the ULB, is computed from \eqref{QF} 
(applied with $L_m(n,s)=N$) to give the right-hand side of \eqref{ulb-general}.
\blue{Note that eventually everything is determined by $n$ and $N$. We will see how it works in practice in Section 5.2.}

We next explain the derivation of the universal upper bound (UUB) from \cite{BDHSS2020} \blue{(see Section 3.2 in that paper)} which is based on choice of polynomials 
   $$
    f(t)=-\lambda f_m^{(n,s)}(t) + g_T(t) 
    $$
\blue{for given $n$, $N$, and $s$. As mentioned in the Introduction, 
the polynomial $f(t)$ has to satisfy $f \geq h$ for $t \in [-1,s]$ 
and to have $f_i \leq 0$ for $i \geq 1$ in its Gegenbauer expansion. 
To fulfil these conditions,} $f_m^{(n,s)}(t)$ is taken to be the degree-$m$ Levenshtein polynomial \eqref{lev-poly},  $g_T(t)$ interpolates the potential function at the multiset $T,$ which consists of the roots of $f_m^{(n,s)}(t)$ 
(counted with their multiplicities; this means that the degree of $g_T$ is $m-1$) and $\lambda=\max\{g_i/\ell_i : 1 \leq i \leq m-1\}$ is a positive constant. More specifically, where 
   $$
    f_m^{(n,s)}(t)=\sum_{i=0}^m \ell_i P_i^{(n)}(t), \ g_T(t)=\sum_{i=0}^{m-1} g_i  P_i^{(n)}(t) 
    $$
are the Gegenbauer expansions of $f_m^{(n,s)}(t)$ and $g_T(t)$, respectively (note that $\ell_i>0$ for every $i \leq m$ \cite[Theorem 5.42]{lev98}). 
The parameter $N_1=L_m(n,s) \geq N$, \blue{computed for given $n$ ans $s$ (the latter determining $m$ uniquely)}, is used \blue{to find the
parameters $\rho_i$ and $\alpha_i$ exactly as in the ULB part (but with $N_1$ instead of $N$; for this to work we assume that $N_1=L_m(n,s) \in [D^*(n,m),D^*(n,m+1))$). Note that} the equality $N_1=N$ holds if and only if there exists a universally optimal code of size $N$ in $n$ dimensions (in this case, ULB and UUB coincide\footnote{Having said that, we may view the difference between the ULB and UUB as a measure of how far the codes are from being universally optimal.}). In our computations of UUBs below we first find the Hermite interpolant $g_T(t)$, then the parameter $\lambda$ \blue{(which already gives $f(t)$)}, and finally compute the bound \eqref{uub-general}.

\subsection{Lower bounds}

\begin{proposition} \label{prop2-1}
For $2\le N\le n+1$ we have
\begin{equation} \label{ulb-deg1}
E_L(n,N) \geq -\tau_1(n,N).
\end{equation}
For $n+1\le N\le 2n$, we have 
\begin{equation} \label{ulb-deg2}
 E_L(n,N) \geq -\tau_2(n,N).
\end{equation}
For $2n\le N \le n(n+3)/2$, we have 
\begin{equation} \label{ulb-deg3}
E_L(n,N) \geq  -\tau_3(n,N).
\end{equation}
where $\tau_1$, $\tau_2$, and $\tau_3$ are defined in \eqref{eq:lb1}-\eqref{eq:lb3}.
\end{proposition}

These estimates constitute the first three bounds in \eqref{ulb-general}, beginning with expressing the 
parameters $(\rho_i,\alpha_i)$ as functions of the dimension $n$ and cardinality $N \in [D^*(n,m),D^*(n,m+1))$, $m=1,2,3$. In all three proofs below we first
find the roots $\alpha_i$ of the Levenshtein polynomial \eqref{lev-poly} setting $L_m(n,s)=N$ for $m=1,2,3$, respectively. 
This is equivalent to solving in $s$ the equation $L_m(n,s)=N$. Then we give the weights $\rho_i$, computed by
setting suitable polynomials (we used $f(t)=1,t,t^2,t^3$; for example $f(t)=1$ gives the identity
$\sum_{i=1}^{k-1+\varepsilon} \rho_i =1-1/N$) in the quadrature formula \eqref{QF}.

{\it Proof of \eqref{ulb-deg1}}.
For the degree 1 bound \eqref{ulb-deg1} we have $\alpha_0=-1/(N-1)$ and $\rho_0 = -1/N\alpha_0 = (N-1)/N$. Therefore
\[ E_L(n,N) \geq N^2\rho_0 L(\alpha_0) = N(N-1)L(\alpha_0)= -N\sqrt{2N(N-1)}. \]
\hfill $\Box$

{\it Proof of \eqref{ulb-deg2}.} For degree 2 (with $k=1$ and $\varepsilon=1$) we have 
$\alpha_0=-1$, $\alpha_1 = -\frac{2n-N}{n(N-2)}$, $\rho_0 = \frac{N-n-1}{Nn+N-4n}$ and 
$\rho_1 = \frac{n(N-2)^2}{N(Nn+N-4n)}$. Since $L(-1)=-2$ and $L(\alpha_1)=-\sqrt{\frac{2N(n-1)}{n(N-2)}}$, we obtain that
the expression $N^2(\rho_0 L(\alpha_0) + \rho_1 L(\alpha_1))$ from \eqref{ulb-general} 
is equal to $-\tau_2(n,N)$ as given in \eqref{eq:lb2}. \hfill $\Box$

{\it Proof of \eqref{ulb-deg3}.}
For the degree-3 lower bound we take $k=2$ and $\varepsilon=0$. By \eqref{ulb-general} we have
\begin{equation} \label{ulb-3}
E_L(n,N) \geq N^2 (\rho_0 L(\alpha_0)+\rho_1 L(\alpha_1)),
\end{equation}
where $N \in [D^{\ast}(n,3),D^{\ast}(n,4)]=[2n,n(n+3)/2]$, and
\[ \alpha_{0,1} = \frac{-n(n-1) \pm \sqrt{D}}{2n(N-n-1)}, \  \ D=n^2(n-1)^2+4n(N-n-1)(N-2n), \]
are the roots of the quadratic equation $n(N-n-1)s^2+n(n-1)s+2n-N=0$ obtained from the equality $L_3(n,s)=N$. 
Further, the weights $\rho_0$ and $\rho_1$ satisfy the formulas 
\[ \rho_0N=\frac{1-\alpha_1^2}{\alpha_0(\alpha_1^2-\alpha_0^2)}, \ \ \rho_1N=\frac{1-\alpha_0^2}{\alpha_1(\alpha_0^2-\alpha_1^2)} \]
(note that the numerators resemble the potential $L(t)$ computed for $\alpha_0,\alpha_1$; this will make
our expressions symmetric). In the sequel, 
we use the following symmetric expressions for $\alpha_0$ and $\alpha_1$
  \begin{gather*}
 \alpha_0+\alpha_1=-\frac{n-1}{N-n-1}, \  \ \alpha_0 \alpha_1=- \frac{N-2n}{n(N-n-1)}, \ \ \alpha_0^2-\alpha_1^2=\frac{(n-1)\sqrt{D}}{n(N-n-1)^2}, \\
 (1-\alpha_0)(1-\alpha_1)=\frac{(n-1)N}{n(N-n-1)}, \ \  (1+\alpha_0)(1+\alpha_1)=\frac{(n-1)(N-2n)}{n(N-n-1)}.
 \end{gather*}

Our task is to express the bound \eqref{ulb-3} via $n$ and $N$. Using the above equalities, we obtain 
    \begin{align*}
E_L(n,N) &\geq N (\rho_0N L(\alpha_0)+\rho_1N L(\alpha_1)) \\
&= -N \left(\frac{(1-\alpha_1^2)\sqrt{2(1-\alpha_0)}}{\alpha_0(\alpha_1^2-\alpha_0^2)}+
\frac{(1-\alpha_0^2)\sqrt{2(1-\alpha_1)}}{\alpha_1(\alpha_0^2-\alpha_1^2)}  \right) \\
=& -\frac{n^2N(N-n-1)^3}{(n-1)(N-2n)\sqrt{D}}\left(\alpha_1(1-\alpha_1^2)\sqrt{2(1-\alpha_0)}-\alpha_0(1-\alpha_0^2)\sqrt{2(1-\alpha_1)}\right).
\end{align*}

Consider the expression $S=\alpha_1(1-\alpha_1^2)\sqrt{2(1-\alpha_0)}-\alpha_0(1-\alpha_0^2)\sqrt{2(1-\alpha_1)}$.
We compute
\begin{align*}
\frac{S^2}{2} &= \frac{(n-1)\left(A-B\right)N}{n(N-n-1)},
\end{align*}            
and thus
   $$
    S=\sqrt{\frac{2(A-B)(n-1)N}{n(N-n-1)}},
    $$
where we have denoted
\begin{align*}
A 
&= \frac{(n-1)(N-2n)^2[Nn^3+(2N-1)n^2-(N-1)(7N-2)n+(N-1)^2(2N+3)]}{n^2(N-n-1)^5}
\end{align*}   
and
\begin{align*}
B 
&= - \frac{2(n-1)(N-2n)^2\sqrt{(n-1)N}}{(n(N-n-1))^{5/2}}.
\end{align*}   

Therefore
\[ E_L(n,N) \geq  -\frac{nN(N-n-1)^2}{(N-2n)\sqrt{D}} \sqrt{\frac{2(A-B)nN(N-n-1)}{n-1}}.  \]
\blue{Performing simplifications under the square root, we obtain
\begin{eqnarray*} 
\frac{2(A-B)nN(N-n-1)}{n-1} &=& 2nN(N-n-1)\left(\frac{(N-2n)^2A_1}
{n^2(N-n-1)^5}+\frac{2(N-2n)^2\sqrt{N(n-1)}}{n^{5/2}(N-n-1)^{5/2}}\right) \\
&=& \frac{2N(N-2n)^2}{n^4(N-n-1)^4}\left(n^3A_1+2\sqrt{N(n-1)n^5(N-n-1)^5}\right) \\
&=& \frac{2N(N-2n)^2}{n^2(N-n-1)^4} \left(nA_1+2(N-n-1)^2B_1\right)
\end{eqnarray*}
with $A_1$ and $B_1$ as in \eqref{eq:A1} and \eqref{eq:B1}, respectively.
Upon substituting this back into the bound for $E_L(n,N)$, we obtain
   \[ 
E_L(n,N) \geq  -\frac{N\sqrt{2N(nA_1+2(N-n-1)^2B_1)}}{\sqrt{D}},  \]}
establishing the bound \eqref{ulb-deg3} with $\tau_3(n,N)$ as in \eqref{eq:lb3}. 
\hfill $\Box$

\subsection{Upper bounds} 
In this section we prove bounds \eqref{eq:ub1}-\eqref{eq:ub3}, deriving an explicit form of the first three universal upper bounds for ${\cC}_N(n,s)$ codes from \cite{BDHSS2020} for $L(t)$ as functions of 
$n$, $N$ and $s$. In addition to the parameters $(\rho_i,\alpha_i)$ as explained above (but now related to $N_1=L_m(n,s)$ instead 
of $N$), we need to find the polynomial $g_T(t)$, then the real parameter  $\lambda$ and finally the polynomial $f(t)$ as explained in the last paragraph of Section 5.1. 
Recall again that because of the sign change, the inequalities \eqref{eq:ub1}-\eqref{eq:ub3} are inverted.

\begin{proposition} \label{prop2-4}
For $N \in [2,n+1]$ and $s \in [-1/(N-1),-1/n]$, we have
\begin{equation} \label{uub-deg1}
E_L(n,N,s) \leq -\tau^{(1)}(n,N,s).
\end{equation}
For $N \in [n+1,2n]$ and $s \in [(N-2n)/n(N-2),0]$, we have 
\begin{equation} \label{uub-deg2}
 E_L(n,N,s) \leq  -\tau^{(2)}(n,N,s).
\end{equation}
For $N \in [2n,n(n+3)/2]$ and $s \in \left[\frac{\sqrt{n^2(n-1)^2+4n(N-n-1)(N-2n)}-n(n-1)}{2n(N-n-1)},\frac{\sqrt{n+3}-1}{n+2}\right]$, we have
\begin{equation} \label{uub-deg3}
\begin{split}
E_L(n,N,s) &\leq -\tau^{(3)}(n,N,s) \\
\end{split}
\end{equation}
where the quantities $\tau^{(1)},\tau^{(2)},\tau^{(3)}$ are defined in \eqref{eq:ub1}-\eqref{eq:ub3} above.
\end{proposition} 

\begin{remark}
We set upper limits for $s$ in all three cases as suggested implicitly by the framework in \cite{BDHSS2020}.  
The bounds are valid beyond these limits but most likely they can be improved by polynomials of higher degrees. 
\end{remark}

{\it Proof of \eqref{uub-deg1}.}
For fixed $n$, $N \in [2,n+1]$ and $s \in [-1/(N-1),-1/n]$, we consider the degree 1 UUB 
\begin{equation*}
E_L(n,N,s) \leq N\left(\frac{N}{L_1(n,s)}-1\right)f(1) +N^2\rho_0 L(s),
\end{equation*}
where the parameters are as follows: $L_1(n,s)=(s-1)/s=: N_1$ is the first Levenshtein bound,
\[ f(t)=-\lambda f_1^{(n,s)}(t)+g_T(t)=-\lambda (t-s)+g_T(t) \]
is our linear programming polynomial, and $\alpha_0=s$, $\rho_0=-1/N_1s =1/(1-s)$
are Levenshtein's parameters corresponding to $s$ (i.e., to $N_1$). The polynomial $g_T(t)$ is constant and is found from 
$g_T(s)=L(s)$. Then $\lambda=0$ and $f(t)=L(s)$ give the bound  
\[ E_L(n,N,s) \leq \left(\frac{N}{N_1}-1\right)NL(s) +N^2\rho_0L(s)=N(N-1)L(s). \]
\qed

\begin{remark} As already observed, this bound is straightforward upon estimating all terms in the energy sum $E_L({\cC}_N)$
by the constant $L(s)$.
\end{remark}

{\it Proof of \eqref{uub-deg2}.}
For fixed $n$, $N \in [n+1,2n]$ and $s \in [(N-2n)/n(N-2),0]$, we consider the degree 2 UUB following the derivation in \cite{BDHSS2020}
\begin{equation} \label{uub-2}
E_L(n,N,s) \leq N\left(\frac{N}{L_2(n,s)}-1\right)f(1) +N^2(\rho_0 L(\alpha_0)+\rho_1 L(\alpha_1)),
\end{equation}
where the parameters are defined as follows: $N_1:=L_2(n,s)=2n(1-s)/(1-ns)$ is the second Levenshtein bound,
\[ f(t)=-\lambda f_2^{(n,s)}(t)+g_T(t)=-\lambda (t+1)(t-s)+g_T(t) \]
is our linear programming polynomial (to be described below), and
\[ \alpha_0=-1, \ \alpha_1=s, \ \rho_0=\frac{N_1-n-1}{N_1 n+N_1-4n}, \ \rho_1=\frac{n(N_1-2)^2}{N_1(N_1 n+N_1-4n)} \]
are the Levenshtein parameters corresponding to $s$ (compare with the parameters in the proof of \eqref{ulb-deg2}).

The polynomial $g_T(t)$ with $T=\{-1,s\}$, i.e. $g_T(-1)=L(-1)$, $g_T(s)=L(s)$, becomes 
\[ g_T(t)=\frac{L(s)-L(-1)}{1+s}t+\frac{L(s)+sL(-1)}{1+s}=\frac{(2-\sqrt{2(1-s)})t-\blue{2s-}\sqrt{2(1-s)}}{1+s}. \]
The coefficient $\lambda$ is chosen to make $f_1=0$ in the Gegenbauer expansion $f(t)=f_2P_2^{(n)}(t)+f_1P_1^{(n)}(t)+f_0$
  (this choice is unique). This gives $\lambda=\frac{2-\sqrt{2(1-s)}}{1-s^2}$ and 
\[ f(t)=-\frac{(2-\sqrt{2(1-s)})t^2-2s^2+\sqrt{2(1-s)}}{1-s^2}, \]
whence $f(1)=-2$. 

Therefore, \eqref{uub-2} gives
\[ E_L(n,N,s) \leq N\left(\frac{N}{N_1}-1\right)(-2) +
N^2\left(\frac{(N_1-n-1)(-2)}{N_1n+N_1-4n}+\frac{n(N_1-2)^2(-\sqrt{2(1-s)})}{N_1(N_1n+N_1-4n)}\right), \] 
implying \eqref{uub-deg2}. \qed

{\it Proof of \eqref{uub-deg3}.} For fixed $n$, $N$, and $s$ as in the condition \eqref{eq:s}, we derive the degree 3 UUB 
\begin{equation} \label{uub-21}
E_L(n,N,s) \leq N\left(\frac{N}{L_3(n,s)}-1\right)f(1) +N^2(\rho_0 L(\alpha_0)+\rho_1 L(\alpha_1)),
\end{equation}
where the parameters are defined as follows:
\[ N_1:=L_3(n,s)=\frac{n(1-s)((n+1)s+2)}{1-ns^2} \]
is the third Levenshtein bound,
\[ f(t)=-\lambda f_3^{(n,s)}(t)+g_T(t)=-\lambda (t-\alpha_0)^2(t-s)+g_T(t) \]
is the linear programming polynomial to be found, and 
\[ \alpha_0=\frac{-n(n-1)-\sqrt{D_1}}{2n(N_1-n-1)} = -\frac{1+s}{1+ns}, \ \alpha_1=\frac{-n(n-1)+\sqrt{D_1}}{2n(N_1-n-1)} = s, \]
\[ D_1=n^2(n-1)^2+4n(N_1-n-1)(N_1-2n) = \frac{n^2(n-1)^2(1+2s+ns^2)^2}{(1-ns^2)^2}, \]
\[ \rho_0=\frac{1-\alpha_1^2}{N_1\alpha_0(\alpha_1^2-\alpha_0^2)} = \frac{(1+ns)^3}{n((n+1)s+2)(1+2s+ns^2)}, \] 
\[ \rho_1=\frac{1-\alpha_0^2}{N_1\alpha_1(\alpha_0^2-\alpha_1^2)} = \frac{n-1}{n(1-s)(1+2s+ns^2)}, \]
are the Levenshtein's parameters corresponding to $s$ (note that they are also shown to depend on $n$ and $s$ only).

The ULB part $\rho_0 L(\alpha_0)+\rho_1 L(\alpha_1)$ in \eqref{uub-21} can be found as in the proof of \eqref{ulb-deg3} \blue{but with $N_1$ instead of $N$. Explicitly, this means that
    $$ 
    \rho_0 L(\alpha_0)+\rho_1 L(\alpha_1)=-\frac{1}{N_1} \sqrt{\frac{2N_1(nA_1+2(N_1-n-1)^2B_1)}{D_1}}, 
    $$
where $A_1$ and $B_1$ are as in \eqref{eq:A1} and \eqref{eq:B1}, respectively, but with $N_1$ instead of $N$, and $D_1$ as above (so $D_1$ has the same form as $D,$ but with $N_1$ instead of $N$).} We obtain
\begin{equation} \label{uub-22}
E_L(n,N,s) \leq \frac{N}{N_1}\left((N-N_1)f(1)-N\sqrt{\frac{2N_1(nA_1+2(N_1-n-1)^2B_1)}{D_1}}\right),
\end{equation}
\blue{In order to rewrite \eqref{uub-22} in terms of $n$ and $s$, we 
first write the ULB part in terms of $n$ and $s$ by using the above expressions, i.e.} 
\begin{eqnarray*}
A_1 &=& \frac{(n-1)^2[(1+ns)^5(1-s)+(n-1)^2((n+1)s+2)]}{(1-ns^2)^3}, \\
B_1 &=& \frac{n(n-1)\sqrt{(1-s)(1+ns)((n+1)s+2)}}{1-ns^2}, \\
N_1-n-1 &=& \frac{(n-1)(1+ns)}{1-ns^2}, 
\end{eqnarray*}
and $D_1=D_1(n,s)$ as found above. We find 
\begin{eqnarray} \label{eq:uub3-ulb-part}
E_L(n,N,s) &\leq& \blue{ \frac{N}{N_1}\left((N-N_1)f(1)-\frac{(1-ns^2)N\sqrt{2N_1(nA_1+2(N_1-n-1)^2B_1)}}{n(n-1)(1+2s+ns^2)}\right)} 
\nonumber \\
&=& \blue{\frac{N}{N_1}\left((N-N_1)f(1)-\frac{N\sqrt{2N_1(nA_2+2(1+ns)^2B_2)}}{n(1+2s+ns^2)}\right)} \nonumber \\
&=& \frac{N}{N_1}\left((N-N_1)f(1)-\frac{N\sqrt{2(1-s)((n+1)s+2)(A_2+2(1+ns)^2B_2)}}{(1+2s+ns^2)(1-ns^2)}\right), 
\end{eqnarray}
\blue{where $A_2$ and $B_2$ are as given in \eqref{eq:AB}. }

Second, we find $f(t)$ in order to compute $f(1)$. The polynomial $g_T(t)=at^2+bt+c$ interpolates $L(t)$ in $T=\{\alpha_0,\alpha_0,\alpha_1\}$, i.e. 
$g(\alpha_0)=L(\alpha_0)$, $g^\prime (\alpha_0)=L^\prime(\alpha_0)$, and $g(\alpha_1)=L(\alpha_1)$.
Resolving this to find $a$, $b$, and $c$, we obtain the Gegenbauer expansion of $f(t)$ as follows
\begin{eqnarray*}
&& f(t)=-\frac{\lambda(n-1)}{n+2} P_3^{(n)}(t)+\frac{(n-1)(a+\lambda(2\alpha_0+\alpha_1))}{n}P_2^{(n)}(t) \\
&& +\, \left(b-\frac{\lambda((\alpha_0^2+2\alpha_0\alpha_1)(n+2)+3)}{n+2}\right)P_1^{(n)}(t) +
\frac{\lambda(\alpha_0^2\alpha_1 n +2\alpha_0+\alpha_1)+a+cn}{n} P_0^{(n)}(t),
\end{eqnarray*}
where 
\begin{eqnarray*}
a &=& \frac{L(\alpha_1)-L(\alpha_0)-L^\prime(\alpha_0)(\alpha_1-\alpha_0)}{(\alpha_1-\alpha_0)^2}, \\
b &=& \frac{L^\prime(\alpha_0)(\alpha_1^2-\alpha_0^2)-2\alpha_0(L(\alpha_1)-L(\alpha_0))}{(\alpha_1-\alpha_0)^2}, \\
c &=& \frac{\alpha_0^2(L(\alpha_1)-L(\alpha_0))-\alpha_0\alpha_1(\alpha_1-\alpha_0)L^\prime(\alpha_0)+(\alpha_1-\alpha_0)^2L(\alpha_0)}
{(\alpha_1-\alpha_0)^2}.
\end{eqnarray*}

According to the rule in Theorem 3.2 from \cite{BDHSS2020}, the coefficient $\lambda$ has to be chosen as $\max\{g_1/\ell_1,g_2/\ell_2\},$ which is equivalent to the choice between 
$\{ f_1=0,f_2<0\}$  and $\{f_1<0,f_2=0\}$, respectively. We will prove below that $f_2<0$, i.e., that the first of these conditions is realized
for all $n$ and $s$ under consideration. 

The equality $f_1=0$ gives 
   \begin{align*}
\lambda&=\frac{b(n+2)}{(\alpha_0^2+2\alpha_0\alpha_1)(n+2)+3} \\
& = \frac{(n+2)(L^\prime(\alpha_0)(\alpha_1^2-\alpha_0^2)-2\alpha_0(L(\alpha_1)-L(\alpha_0)))}
{(\alpha_1-\alpha_0)^2((\alpha_0^2+2\alpha_0\alpha_1)(n+2)+3)}. 
   \end{align*}
Then 
\[ f(1) = -\lambda (1-\alpha_0)^2(1-\alpha_1)+a+b+c = \frac{A_3 (L(\alpha_1)-L(\alpha_0))+B_3 L(\alpha_0)-C_3 L^\prime(\alpha_0)}{B_3}, \]
where
\begin{eqnarray*}
A_3 &=& (1-\alpha_0)^2((n+2)(1+\alpha_0)^2-n+1) = \frac{(n-1)((n+1)s+2)^2((n-2)s^2-2ns-1)}{(1+ns)^4}, \\
B_3 &=& (\alpha_1-\alpha_0)^2((\alpha_0^2+2\alpha_0\alpha_1)(n+2)+3) \\
    &=& -\frac{(1+2s+ns^2)^2(2n(n+2)s^3-(n^2-5n-2)s^2-6ns-n-5)}{(1+ns)^4}, \\
C_3 &=& (1-\alpha_0)(1-\alpha_1)(\alpha_1-\alpha_0)((n+2)(\alpha_0+\alpha_1+\alpha_0\alpha_1)+3) \\
    &=& \frac{(n-1)(1-s)((n+1)s+2)(1+2s+ns^2)((n+2)s^2+2s-1)}{(1+ns)^3}. \\
\end{eqnarray*}
Therefore
\[ f(1) = \frac{((n+1)s+2)\left[(1-s)(1+ns)A_4 + B_4\sqrt{(1-s)B_5}\right]} {(1+2s+ns^2)^2C_4\sqrt{2B_5}}, \]
where $A_4$, $B_4$, $B_5$ and $C_4$ are as given in Equation \eqref{eq:AB} in Section 2.

Substituting these parameters into \eqref{eq:uub3-ulb-part} and performing simplifications, we 
eventually obtain \eqref{eq:ub3}:
\blue{
\begin{eqnarray*}
E_L(n,N,s) &\leq& \frac{N}{N_1}\left((N-N_1)f(1)-\frac{N\sqrt{2(1-s)((n+1)s+2)(A_2+2(1+ns)^2B_2)}}{(1+2s+ns^2)(1-ns^2)}\right), \\
&=& \frac{N(1-ns^2)}{n(1-s)((n+1)s+2)}\left(\frac{A_5f(1)}{1-ns^2}-\frac{N\sqrt{2(1-s)((n+1)s+2)(A_2+2(1+ns)^2B_2)}}{(1+2s+ns^2)(1-ns^2)}\right), \\
&=& \frac{NA_5((1-s)(1+ns)A_4+B_4\sqrt{(1-s)B_5})}{n(1-s)(1+2s+ns^2)^2C_4 \sqrt{2B_5}}-\frac{N^2\sqrt{2(1-s)((n+1)s+2)(A_2+2(1+ns)^2B_2)}}{n(1-s)((n+1)s+2)(1+2s+ns^2)}, \\
&=& \frac{NA_5((1-s)(1+ns)A_4+B_4\sqrt{(1-s)B_5})}{n(1-s)(1+2s+ns^2)^2C_4 \sqrt{2B_5}}-\frac{N^2\sqrt{2(1-s)(A_2+2(1+ns)^2B_2)}}{n(1-s)(1+2s+ns^2)\sqrt{B_5}}, \\
&=& \frac{N[A_5((1-s)(1+ns)A_4+B_4\sqrt{(1-s)B_5}) - 2N(1+2s+ns^2)C_4\sqrt{A_6}]}{n(1-s)(1+2s+ns^2)^2C_4\sqrt{2B_5}},
\end{eqnarray*}
}
\blue{where $A_i$, $B_i$, and $C_i$ are as given in \eqref{eq:AB}. }

The condition $f_2<0$ is equivalent to $\lambda(2\alpha_0+s) + a < 0$. This gives the inequality
\[ \frac{6B_6\sqrt{(1-s)(1+ns)((n+1)s+2)} - C_6}{C_4} < 0, \]
where
\begin{eqnarray*}
 B_6 &=& (n-2)(n+1)s^2-4s-n-1, \\
 C_6 &=& n^3(n+2)s^6 + 3n^2(n+2)s^5 - 3n(n^2-n-2)s^4 + 2(3n^3-6n^2-8n-4)s^3 + \\ 
 && 3(3n^2-16n-14)s^2 - 3(2n^2+5n+18)s - 11n - 13. \\
\end{eqnarray*}
We have $C_4<0$ since $2n(n+2)s^3<n+5$ follows for $n \geq 3$ and $0<s<(-1+\sqrt{n+3})/(n+2)$ (just use that $s<1/\sqrt{n+2}$). 
It remains to see that $6B_6\sqrt{(1-s)(1+ns)((n+1)s+2)} > C_6$. Since $B_6<0$ for $0<s<(-1+\sqrt{n+3})/(n+2)$, we need to prove 
that $C_6^2> 36B_6^2 (1-s)(1+ns)((n+1)s+2)$. This inequality is reduced to an 8-degree polynomial (in $s$) inequality shown to hold 
true by a computer algebra system.  \qed

{\bf Acknowledgements.} We thank the referees for spending time and effort on our manuscript and for providing an extensive set of remarks which improved the exposition.

\providecommand{\bysame}{\leavevmode\hbox to3em{\hrulefill}\thinspace}
\providecommand{\MR}{\relax\ifhmode\unskip\space\fi MR }
\providecommand{\MRhref}[2]{%
  \href{http://www.ams.org/mathscinet-getitem?mr=#1}{#2}
}
\providecommand{\href}[2]{#2}


\begin{thebibliography}{10}

\bibitem{Alexander1972}
J.~R. Alexander, \emph{On the sum of distances between $n$ points on the
  sphere}, Acta Math. Hungar. \textbf{23} (1972), no.~3-4, 443--448.

\bibitem{Balla2018}
I.~Balla, F.~Dr\"{a}xler, P.~Keevash, and B.~Sudakov, \emph{Equiangular lines
  and spherical codes in {E}uclidean space}, Invent. Math. \textbf{211} (2018),
  no.~1, 179--212.

\bibitem{BB2005}
E.~Bannai and Et. Bannai, \emph{A note on the spherical embeddings of strongly
  regular graphs}, European J. Comb. \textbf{26} (2005), 1177--1179.

\bibitem{Barg2021}
A.~Barg, \emph{Stolarsky's invariance principle for finite metric spaces},
  Mathematika \textbf{67} (2021), no.~1, 158--186.

\bibitem{BGOY15}
A.~Barg, A.~Glazyrin, K.~A. Okoudjou, and W.-H. Yu, \emph{Finite two-distance
  tight frames}, Linear Algebra Appl. \textbf{475} (2015), 163--175.

\bibitem{Beck1984}
J.~Beck, \emph{Sums of distances between points on a sphere---an application of
  the theory of irregularities of distribution to discrete geometry},
  Mathematika \textbf{31} (1984), no.~1, 33--41.

\bibitem{Benedetto2003}
J.J. Benedetto and M.~Fickus, \emph{Finite normalized tight frames}, Adv.
  Comput. Math. \textbf{18} (2003), no.~2--4, 357--385.

\bibitem{Bilyk2018}
D.~Bilyk, F.~Dai, and R.~Matzke, \emph{The {S}tolarsky principle and energy
  optimization on the sphere}, Constr. Approx. \textbf{48} (2018), 31--60.

\bibitem{BilykLacey2017}
D.~Bilyk and M.T. Lacey, \emph{One-bit sensing, discrepancy and {S}tolarsky's
  principle}, Sbornik: Mathematics \textbf{208} (2017), no.~6, 744--763,
  Translation from the Russian, {\em Mat. Sbornik}, vol. 208, no.~6, pp.4--25.

\bibitem{BilykMatzke2019}
D.~Bilyk and R.~W. Matzke, \emph{On the {F}ejes {T}\'{o}th problem about the
  sum of angles between lines}, Proc. Amer. Math. Soc. \textbf{147} (2019),
  no.~1, 51--59.

\bibitem{BHS2019}
S.~V. Borodachov, D.~P. Hardin, and E.~B. Saff, \emph{Discrete energy on
  rectifiable sets}, Springer, 2019.

\bibitem{BDL1999}
P.~Boyvalenkov, D.~Danev, and I.~Landjev, \emph{On maximal spherical codes
  {I}{I}}, J. Combin. Des. \textbf{7} (1999), 316--326.

\bibitem{BDHSS2015}
P.~G. Boyvalenkov, P.~D. Dragnev, D.~P. Hardin, E.~B. Saff, and M.~M.
  Stoyanova, \emph{Universal upper and lower bounds on energy of spherical
  designs}, Dolomites Res. Notes Approx. \textbf{8} (2015), 51--65.

\bibitem{BDHSS2016}
\bysame, \emph{Universal lower bounds for potential energy of spherical codes},
  Constr. Approx. \textbf{44} (2016), no.~3, 385--415.

\bibitem{BDHSS2019}
\bysame, \emph{Energy bounds for codes in polynomial metric spaces}, Anal.
  Math. Phys. \textbf{9} (2019), no.~2, 781--808.

\bibitem{BDHSS2020}
\bysame, \emph{Upper bounds for energies of spherical codes of given
  cardinality and separation}, Des. Codes Cryptogr. \textbf{88} (2020), no.~9,
  1811--1826.

\bibitem{Brauchart2012}
J.~S. Brauchart, D.~P. Hardin, and E.~B. Saff, \emph{The next-order term for
  optimal {R}iesz and logarithmic energy asymptotics on the sphere}, Recent
  advances in orthogonal polynomials, special functions, and their
  applications, Contemp. Math., vol. 578, Amer. Math. Soc., Providence, RI,
  2012, pp.~31--61.

\bibitem{Brauchart2013}
J.S. Brauchart and J.~Dick, \emph{A simple proof of {S}tolarsky's invariance
  principle}, Proc. Amer. Math. Soc. \textbf{141} (2013), 2085--2096.

\bibitem{Brouwer2022}
A.~E. Brouwer and H.~Van~Maldeghem, \emph{Strongly regular graphs},
  Encyclopedia of Mathematics and its Applications, vol. 182, Cambridge
  University Press, Cambridge, 2022.

\bibitem{Brouwer}
A.E. Brouwer, \emph{{SRG} family parameters},
  \href{https://www.win.tue.nl/~aeb/graphs/srghub.html}{https://www.win.tue.nl/~aeb/graphs/srghub.html}.

\bibitem{Brouwer2012}
A.E. Brouwer and W.H. Haemers, \emph{Spectra of graphs}, Springer, New York,
  2012.

\bibitem{CK1986}
R.~Calderbank and W.~M. Kantor, \emph{The geometry of two-weight codes}, Bull.
  London Math. Soc. \textbf{18} (1986), no.~2, 97--122.

\bibitem{CGS1978}
P.J. Cameron, J.M. Goethals, and J.J. Seidel, \emph{Strongly regular graphs
  having strongly regular subconstituents}, J. Algebra \textbf{55} (1978),
  257--280.

\bibitem{coh07b}
H.~Cohn and A.~Kumar, \emph{Universally optimal distribution of points on
  spheres}, J. Amer. Math. Soc. \textbf{20} (2007), no.~1, 99--148.

\bibitem{deCaen2000}
D.~de~Caen, \emph{Large equiangular sets of lines in {E}uclidean space},
  Electron. J. Combin. \textbf{7} (2000), Research Paper 55, 3.

\bibitem{del77b}
P.~Delsarte, J.~M. Goethals, and J.~J. Seidel, \emph{Spherical codes and
  designs}, Geometriae Dedicata \textbf{6} (1977), 363--388.

\bibitem{D2016}
C.~Ding, \emph{A construction of binary linear codes from {B}oolean functions},
  Discrete Math. \textbf{339} (2016), no.~9, 2288--2303.

\bibitem{Ericson2001}
T.~Ericson and V.~Zinoviev, \emph{Codes on {E}uclidean spheres}, Elsevier
  Science, Amsterdam e.a., 2001.

\bibitem{FejesToth1956}
L.~Fejes~T\'{o}th, \emph{On the sum of distances determined by a pointset},
  Acta Math. Acad. Sci. Hungar. \textbf{7} (1956), 397--401.

\bibitem{GlazyrinYu2018}
A.~Glazyrin and W.-H. Yu, \emph{Upper bounds for {$s$}-distance sets and
  equiangular lines}, Adv. Math. \textbf{330} (2018), 810--833.

\bibitem{GodsilRoyle2001}
C.~Godsil and G.~Royle, \emph{Algebraic graph theory}, Graduate Texts in
  Mathematics, vol. 207, Springer-Verlag, New York, 2001.

\bibitem{HouShao2011}
X.~Hou and J.~Shao, \emph{Spherical distribution of 5 points with maximal
  distance sum}, Discrete Comput. Geom. \textbf{46} (2011), no.~1, 156--174.

\bibitem{Jedwab2015}
J.~Jedwab and A.~Wiebe, \emph{Large sets of complex and real equiangular
  lines}, J. Combin. Theory Ser. A \textbf{134} (2015), 98--102.

\bibitem{Jiang2020a}
Z.~Jiang and A.~Polyanskii, \emph{Forbidden subgraphs for graphs of bounded
  spectral radius, with applications to equiangular lines}, Israel J. Math.
  \textbf{236} (2020), no.~1, 393--421.

\bibitem{Jiang2021}
Z.~Jiang, J.~Tidor, Y.~Yao, S.~Zhang, and Y.~Zhao, \emph{Equiangular lines with
  a fixed angle}, Ann. of Math. (2) \textbf{194} (2021), no.~3, 729--743.

\bibitem{Kolushov1997}
A.~V. Kolushov and V.~A. Yudin, \emph{Extremal dispositions of points on the
  sphere}, Anal. Math. \textbf{23} (1997), no.~1, 25--34.

\bibitem{Kuijlaars1998}
A.~B.~J. Kuijlaars and E.~B. Saff, \emph{Asymptotics for minimal discrete
  energy on the sphere}, Trans. Amer. Math. Soc. \textbf{350} (1998), no.~2,
  523--538.

\bibitem{LemmensSeidel1973}
P.~W.~H. Lemmens and J.~J. Seidel, \emph{Equiangular lines}, J. Algebra
  \textbf{24} (1973), 494--512.

\bibitem{lev83a}
V.~I. Levenshtein, \emph{Bounds for packings of metric spaces and some of their
  applications}, Problemy Kibernet. \textbf{40} (1983), 43--110 (In Russian).

\bibitem{lev92}
\bysame, \emph{Designs as maximum codes in polynomial metric spaces}, Acta
  Applicandae Math. \textbf{25} (1992), 1--82.

\bibitem{lev98}
\bysame, \emph{Universal bounds for codes and designs}, Handbook of Coding
  Theory (V.~Pless and W.~C. Huffman, eds.), vol.~1, Elsevier Science,
  Amsterdam, 1998, pp.~499--648.

\bibitem{LLHQ2021}
K.~Li, C.~Li, T.~Helleseth, and L.~Qu, \emph{Binary linear codes with few
  weights from two-to-one functions}, IEEE Trans. Inform. Theory \textbf{67}
  (2021), no.~7, 4263--4275.

\bibitem{mac91}
F.~J. MacWilliams and N.~J.~A. Sloane, \emph{The theory of error-correcting
  codes}, North-Holland, Amsterdam, 1991.

\bibitem{Neumaier1989}
A.~Neumaier, \emph{Graph representations, two-distance sets, and equiangular
  lines}, Linear Algebra Appl. \textbf{114/115} (1989), 141--156.

\bibitem{Sid1971}
V.M. Sidel'nikov, \emph{The mutual correlation of sequences}, Soviet Math.
  Dokl. \textbf{12} (1971), 197--201.

\bibitem{Skriganov2019}
M.~M. Skriganov, \emph{Point distributions in two-point homogeneous spaces},
  Mathematika \textbf{65} (2019), no.~3, 557--587.

\bibitem{Skriganov2020}
\bysame, \emph{Stolarsky's invariance principle for projective spaces}, J.
  Complexity \textbf{56} (2020), 101428.

\bibitem{Stolarsky1973}
K.~B. Stolarsky, \emph{Sums of distances between points on a sphere, {II}},
  Proc. AMS \textbf{41} (1973), 575--582.

\bibitem{Waldron2009}
S.~Waldron, \emph{On the construction of equiangular frames from graphs},
  Linear Algebra Appl. \textbf{431} (2009), no.~11, 2228--2242.

\bibitem{Waldron2018}
S.~F.~D. Waldron, \emph{An introduction to finite tight frames}, Applied and
  Numerical Harmonic Analysis, Birkh\"{a}user/Springer, New York, 2018.

\bibitem{WZHZ2015}
X.~Wang, D.~Zheng, L.~Hu, and X.~Zeng, \emph{The weight distributions of two
  classes of binary cyclic codes}, Finite Fields Appl. \textbf{34} (2015),
  192--207.

\bibitem{Yudin1993}
V.~A. Yudin, \emph{The minimum of potential energy of a system of point
  charges}, Discrete Math. Appl. \textbf{3} (1993), no.~1, 75--81, Translation
  from the Russian, {\em Diskretnaya Matematika}, vol.4, no. 2, 1992, pp.
  115--121.

\end{thebibliography}
\end{document}